\documentclass[12pt]{amsproc}
\usepackage{format}

\title{Arthur's Conjectures and the Orbit Method for Real Reductive Groups}

\begin{document}
\maketitle

\begin{abstract}
    The first half of this article is expository---I will review, with examples, the main statements of the Langlands classification and Arthur's conjectures for real reductive groups as formulated in \cite{AdamsBarbaschVogan}. In the second half, I will turn my attention to the Orbit Method, a conjectural scheme for classifying irreducible unitary representations of a real reductive group. I will give a definition of the Orbit Method in the case when the group is complex. The main input is the theory of unipotent ideals and Harish-Chandra bimodules, developed in \cite{LMBM}. I will show that the Orbit Method I define is related to 
    Arthur's conjectures via a natural duality map. Finally, I will sketch a possible generalization of this Orbit Method for arbitrary real groups. 
\end{abstract}

\section{Introduction}

Let $G$ be a complex connected reductive algebraic group and let $\sigma$ be a real form of $G$ (i.e. an anti-holomorphic involution $\sigma: G \to G$). Consider the set
$$\Pi(G,\sigma) := \{\text{equivalence classes of irreducible admissible representations of } G^{\sigma}\}.$$
The local aim over $\RR$ of the Langlands correspondence is to parameterize this set. This problem was solved completely by Langlands, Knapp, and Zuckerman. More precisely, Langlands in \cite{Langlands} introduced some finite sets called $L$\emph{-packets} which partition $\Pi(G,\sigma)$. The representations in these packets were later parameterized by the work of Knapp and Zuckerman (\cite{KnappZuckerman1982}). Taken together, these results provide a complete paramaterization of $\Pi(G,\sigma)$, called the \emph{Langlands classification for real reductive groups}. In Section \ref{sec:LLCclassical}, I will recall the main statements of this classification, as formulated in \cite{AdamsBarbaschVogan}. 

Each irreducible representation $\pi \in \Pi(G,\sigma)$ can be realized as the unique irreducible quotient of a \emph{standard representation} (this is a representation which is induced from a certain (limit of) discrete series representation of a Levi subgroup of $G^{\sigma}$).
The standard and irreducible representations of $G^{\sigma}$ form two distinct bases, in natural bijection, for the Grothendieck group of finite-length representations of $G^{\sigma}$. A problem which is distinct from, but related to, the classification of $\Pi(G,\sigma)$ is to compute the matrix which relates these two bases (it turns out this problem is essentially equivalent to the problem of computing irreducible characters, see Section \ref{sec:ABV} below). The classical formulation of the Langlands classification does not address this problem (and it is not well-suited for doing so). In \cite{AdamsBarbaschVogan}, Adams, Barbasch, and Vogan introduced a reformulation of the Langlands classification which relates this change of basis matrix to the geometry of a certain algebraic variety of modified Langlands parameters. 

An irreducible admissible representation $\pi \in \Pi(G,\sigma)$ is \emph{unitary} if it admits a positive-definite invariant Hermitian form. Write
$$\Pi_u(G, \sigma) \subset \Pi(G,\sigma)$$
for the subset of irreducible unitary representations of $G^{\sigma}$. A basic problem in representation theory is to identify this subset. This problem is open---it is perhaps the most fundamental unsolved problem in the representation theory of Lie groups. There are two main paradigms for approaching this problem: \emph{Arthur's conjectures} and \emph{the Orbit Method}. Both seek to parameterize the unitary dual (or at least, a large chunk of it) in terms of Lie-theoretic data. 

In a bit more detail, both approaches predict the existence of certain finite sets (or packets) of irreducible unitary representations, which should satisfy various desiderata and should exhaust (a large portion of) $\Pi_u(G,\sigma)$. In the case of Arthur's conjectures, these packets are indexed by so-called \emph{Arthur parameters}---certain Lie-theoretic objects related to the Langlands dual group. In the case of the Orbit Method, these packets are indexed by (roughly speaking) finite covers of co-adjoint $G^{\sigma}$-orbits. 

In \cite{AdamsBarbaschVogan}, Adams, Barbasch, and Vogan proposed a definition of Arthur packets for arbitrary groups involving microlocal geometry on their space of Langlands parameters. Their packets have been shown to possess all of the expected properties, \emph{except for unitarity}, which remains an open question. I will review their construction in Section \ref{sec:Arthur}.

The Orbit Method picture is much less developed. The corresponding packets have not yet been defined in general. In Section \ref{sec:orbit}, I will define them in the case when $G^{\sigma}$ is complex. I will also sketch a generalization for arbitrary groups. These definitions are based on recent joint work with Ivan Losev and Dmitryo Matvieievksyi (\cite{LMBM}) on filtered quantizations of nilpotent covers.

There are some structural similarities between the predictions of Arthur's conjectures and those of the Orbit Method. So it is natural to look for a relationship between these two points of view. In that direction, I will define for complex groups a natural duality map
$$\mathsf{D}: \{\text{Arthur parameters}\} \to \{\text{covers of co-adjoint orbits}\} $$
which `intertwines' Arthur's conjectures and the Orbit Method, see Section \ref{sec:duality}.  


\subsection{Acknowledgments}

I would like to thank David Vogan for many helpful discussions about the orbit method and Jeffrey Adams for reviewing an earlier draft of this article. 
\section{Langlands classification}\label{sec:Langlands}

\subsection{Langlands classification: classical formulation}\label{sec:LLCclassical}

Let $G^{\vee}$ be the (complex reductive) dual group of $G$, and let $G^L$ be the $L$-group associated to the inner class of $\sigma$, see \cite[Section 2.1]{Borel1979}. If we write $\Gamma$ for the Galois group of $\CC/\RR$, there is a short exact sequence of algebraic groups
$$1 \to G^{\vee} \to G^L \to \Gamma \to 1$$
Let $W_{\RR}$ be the Weil group of $\RR$. Recall that $W_{\RR}$ is a (disconnected) Lie group generated by $\CC^{\times}$ and a single element $j$, subject to the relations
$$j^2 = 1, \qquad jzj = \bar{z}$$
So there is a short exact sequence of Lie groups
$$1 \to \CC^{\times} \to W_{\RR} \to \Gamma \to 1.$$
\begin{definition}\label{def:Lparameter}
A Langlands parameter for $G^L$ is a continuous homomorphism
$$\varphi: W_{\RR} \to G^L$$
such that 
\begin{itemize}
    \item[(i)] the image of $\varphi$ consists of semisimple elements, and
    \item[(ii)] the following diagram commutes
\begin{center}
    \begin{tikzcd}
    W_{\RR} \ar[rr,"\varphi"] \ar[dr] & & G^L\ar[dl]\\
    & \Gamma 
    \end{tikzcd}
\end{center}
\end{itemize}
Write $P(G^L)$ for the set of Langlands parameters for $G^L$ and write $\Phi(G^L)$ for the set of $G^{\vee}$-conjugacy classes in $P(G^L)$. A Langlands parameter is \emph{tempered} if the image of $\varphi$ is bounded. Write $P_{\mathrm{temp}}(G^L)$ for the set of tempered Langlands parameters and $\Phi_{\mathrm{temp}}(G^L)$ for the set of $G^{\vee}$-conjugacy classes in $P_{\mathrm{temp}}(G^L)$.
\end{definition}

Since $W_{\RR}$ is generated by $\CC^{\times}$ and $j$, a Langlands parameter is determined by its differential at the identity and its value at $j$. Its differential at the identity is given by a pair of commuting semisimple elements in $\fg^{\vee}$ (encoding the holomorphic and anti-holomorphic components of the differential) and its value at $j$ is a semisimple element in the non-identity component of $G^L$. The requirement that these elements determine a group homomorphism imposes a small number of additional conditions. The result of this analysis is a very concrete Lie-theoretic description of Langlands parameters.

\begin{prop}[Proposition 5.6, \cite{AdamsBarbaschVogan}]\label{prop:params}
There is a natural $G^{\vee}$-equivariant bijection between $P(G^L)$ and pairs $(y,\lambda)$ such that
\begin{itemize}
    \item[(i)] $y$ is a semisimple element in $G^L \setminus G^{\vee}$.
    \item[(ii)] $\lambda$ is a semisimple element in $\fg^{\vee}$.
    \item[(iii)] $y^2=\exp(2\pi i \lambda)$.
    \item[(iv)] $[\lambda, \Ad(y)\lambda]=0$.
\end{itemize}
The map $\varphi \mapsto (\lambda(\varphi),y(\varphi))$ is defined as follows: $\lambda(\varphi)$ is the holomorphic part of $d\varphi: \CC \to \fg^{\vee}$ and $y(\varphi) = \exp(\pi i \lambda(\varphi))\phi(j)$. 

\end{prop}

\begin{theorem}[\cite{Langlands}]\label{thm:LLCclassical}
For each Langlands parameter $\varphi \in P(G^L)$, there is an associated finite (possibly empty) set $\Pi_{\varphi}(G,\sigma) \subset \Pi(G,\sigma)$ called the \emph{L-packet} for $\varphi$ such that
\begin{itemize}
    \item[(i)] Two packets are equal (if parameters are conjugate) or disjoint (otherwise).
    \item[(ii)] The packets partition $\Pi(G,\sigma)$.
\end{itemize}
\end{theorem}

The construction of the packets $\Pi_{\varphi}(G,\sigma)$ is complicated and technical, we refer the reader to \cite[Section 3]{Langlands} or \cite[Section III.10]{Borel1979}. The construction partially reduces (via parabolic induction and formation of discrete series) to the construction for tori, where it can be verified by hand. 

The next proposition describes the constistuents of $\Pi_{\varphi}(G,\sigma)$ in terms of the pair $(y,\lambda)$. It is an easy consequence of the construction of $\Pi_{\varphi}(G,\sigma)$. 

\begin{prop}\label{prop:propsofylambda}
Suppose $\varphi \in P(G^L)$ and let $(y,\lambda) = (y(\varphi),\lambda(\varphi))$ be the corresponding pair (cf. Proposition \ref{prop:params}). Choose maximal torus $H^{\vee} \subset G^{\vee}$ such that $\lambda \in \mathfrak{h}^{\vee}$ and $\Ad(y)$ normalizes $H^{\vee}$. Then the following are equivalent
\begin{itemize}
    \item[(i)] some representation in $\Pi_{\varphi}(G,\sigma)$ is a (limit of) discrete series.
    \item[(ii)] all representations in $\Pi_{\varphi}(G,\sigma)$ are (limits of) discrete series.
    \item[(iii)] $\Ad(y)$ acts by inversion on $H^{\vee}$.
\end{itemize}
and the following are equivalent
\begin{itemize}
    \item[(i)] some representation in $\Pi_{\varphi}(G,\sigma)$ is tempered.
    \item[(ii)] all representations in $\Pi_{\varphi}(G,\sigma)$ are tempered.
    \item[(iii)] $\lambda+\Ad(y)\lambda \in X_*(H^{\vee}) \otimes_{\ZZ} i\RR$.
    \item[(iv)] $\varphi$ is tempered (cf. Definition \ref{def:Lparameter}).
\end{itemize}
Note that $\lambda \in \fh^{\vee} \simeq \fh^*$ determines an infinitesimal character for $\fg$ by means of the Harish-Chandra isomorphism. All of the representations in $\Pi_{\varphi}(G,\sigma)$ have infinitesimal character $\lambda$.
\end{prop}

\begin{example}\label{ex:PGL1}
Let $G = \mathrm{PGL}(2,\CC)$ and let $\sigma_s$ be the split real form of $G$. Then $G^{\sigma_s} = \mathrm{PGL}(2,\RR)$, $G^{\vee} = \mathrm{SL}(2,\CC)$, and $G^L = G^{\vee} \times \Gamma$. Let $H^{\vee} \subset G^{\vee}$ be the diagonal torus. Then
$$N_{G^{\vee}}(H^{\vee}) = \{\begin{pmatrix}x& 0\\0 & x^{-1} \end{pmatrix}\} \bigsqcup\{ \begin{pmatrix}0 & x^{-1}\\-x & 0 \end{pmatrix}\}$$
By Proposition \ref{prop:params}, the set $\Phi(G^L)$ of $G^{\vee}$-conjugacy classes of Langlands parameters is in bijection with pairs
$$(y,\lambda) \in N_{G^{\vee}}(H^{\vee}) \times  \fh^{\vee}, \qquad y^2 = \exp(2\pi i \lambda)$$
up to conjugation by $G^{\vee}$. In the table below, we list all such pairs. For each $(y,\lambda)$ we record the corresponding $L$-packet in $\Pi(G,\sigma)$ (note: each $L$-packet is a singleton).

\begin{center}
    \begin{tabular}{|l|l|l|} \hline
        $\lambda$ & $y$ & $\Pi_{\varphi}(G,\sigma)$\\ \hline
         $\begin{pmatrix}a& 0\\0 & -a \end{pmatrix}$ for $a = \frac{1}{2}, \frac{3}{2},...$ & $\begin{pmatrix}i & 0\\0 & -i\end{pmatrix}$ &  spherical finite-dimensional \\ 
         & & of dimension $a+1/2$ \\ \hline
         
         $\begin{pmatrix}a& 0\\0 & -a \end{pmatrix}$ for $a = \frac{1}{2}, \frac{3}{2},...$ & $\begin{pmatrix}-i & 0\\0 & i\end{pmatrix}$ & non-spherical finite-dimensional \\
         & & of dimension $a+1/2$ \\ \hline
         
        $\begin{pmatrix}a& 0\\0 & -a \end{pmatrix}$ for $a = \frac{1}{2}, \frac{3}{2},...$ & $\begin{pmatrix}0 & 1\\-1 & 0\end{pmatrix}$ & discrete series \\ 
        & & of infinitesimal character $a$ \\ \hline
        
         $\begin{pmatrix}a& 0\\0 & -a \end{pmatrix}$ for $a \notin \{\frac{1}{2},\frac{3}{2},...\}$ & $\begin{pmatrix}\exp(\pi i a) & 0\\0 & \exp(-\pi i a)\end{pmatrix}$ & spherical principal series \\ 
         & & of infinitesimal character $a$ \\ \hline
         
        $\begin{pmatrix}a& 0\\0 & -a \end{pmatrix}$ for $a \notin \{\frac{1}{2},\frac{3}{2},...\}$  & $\begin{pmatrix}-\exp(\pi i a) & 0\\0 & -\exp(-\pi i a)\end{pmatrix}$ & non-spherical principal series \\ 
        & & of infinitesimal character $a$ \\ \hline
         \end{tabular}
\end{center}
Note: to avoid double-counting, one should assume in the final two rows that $\mathrm{Re}(a) \geq 0$ and $\mathrm{Im}(a) \geq 0$ if $\mathrm{Re}(a)=0$. 

\end{example}

Strictly speaking, Theorem \ref{thm:LLCclassical} is \emph{not} a parameterization of $\Pi(G,\sigma)$. Rather, it is a parameterization of the set of $L$-packets, which partition the set $\Pi(G,\sigma)$. This issue can be resolved by considering a finer set of parameters. Assume $G$ is adjoint to simplify the statements (if $G$ is not adjoint, there are several complications, which we would prefer to avoid in this survey. For details on the general case, we refer the reader to \cite{AdamsBarbaschVogan}).

\begin{definition}[Definition 5.11, \cite{AdamsBarbaschVogan}]
A \emph{complete Langlands parameter} for $G^L$ is a pair $(\varphi,\tau)$ consisting of a Langlands parameter $\varphi \in P(G^L)$ and an irreducible representation $\tau$ of the component group of $Z_{G^{\vee}}(\varphi)$. Write $P^c(G^L)$ for the set of complete Langlands parameters for $G^L$ and $\Phi^c(G^L)$ for the set of $G^{\vee}$-conjugacy classes in $P^c(G^L)$.
\end{definition}

For the cleanest formulation of the Langlands classification, we must consider several real forms of $G$ at once. Write $[\sigma]$ for the inner class of $\sigma$ and form the disjoint union
$$\Pi(G,[\sigma]) := \bigsqcup_{\sigma' \in [\sigma]} \Pi(G,\sigma')$$
\begin{theorem}[Theorems 12.3, 12.9, \cite{AdamsBarbaschVogan}]\label{thm:LLCclassicalcomplete}
There is a bijective correspondence
$$\pi: \Phi^c(G^L) \xrightarrow{\sim} \Pi(G,[\sigma])$$
For a fixed real form $\sigma$ and Langlands parameter $\varphi$, we have
$$\Pi_{\varphi}(G,\sigma) = \pi\{(\varphi,\tau)\} \cap \Pi(G,\sigma).$$
\end{theorem}

\begin{example}\label{ex:PGL2}
Let $G=\mathrm{PGL}_2(\CC)$ and let $[\sigma]$ be the unique inner class. We will use the notation of Example \ref{ex:PGL1}. There are two real forms in $[\sigma]$, the split real form $\sigma_s$ and the compact real form $\sigma_c$. The corresponding Lie groups are $G_s = G^{\sigma_s} = \mathrm{PGL}(2,\RR)$ and $G_c = G^{\sigma_c} = \mathrm{SU}(2)/\{\pm \mathrm{Id}\} \simeq \mathrm{SO}(3,\RR)$. We computed the $L$-packets for $\sigma_s$ in Example \ref{ex:PGL1}. Below we compute the $L$-packets for $\sigma_c$ to illustrate Theorem \ref{thm:LLCclassicalcomplete} (note: each $L$-packet is a singleton).

\begin{table}[H]
\tiny
    \begin{tabular}{|l|l|l|l|l|} \hline
        $\lambda$ & $y$ & $\Pi_{\varphi}(G_s)$ & $\Pi_{\varphi}(G_c)$ & $Z_{G^{\vee}}(y,\lambda)$\\ \hline
          $\begin{pmatrix}a& 0\\0 & -a \end{pmatrix}$ for $a = \frac{1}{2}, \frac{3}{2},...$ & $\begin{pmatrix}i & 0\\0 & -i\end{pmatrix}$ & spherical fin-dim& $\emptyset$ & $H^{\vee}$\\
          & &  of dim $a+1/2$ & & \\ \hline
         
          $\begin{pmatrix}a& 0\\0 & -a \end{pmatrix}$ for $a = \frac{1}{2}, \frac{3}{2},...$ & $\begin{pmatrix}-i & 0\\0 & i\end{pmatrix}$ & non-spherical fin-dim& $\emptyset$ & $H^{\vee}$\\ 
          & & of dim $a+1/2 \ \otimes \mathrm{sgn}$ & & \\ \hline
         
         $\begin{pmatrix}a& 0\\0 & -a \end{pmatrix}$ for $a = \frac{1}{2}, \frac{3}{2},...$ & $\begin{pmatrix}0 & 1\\-1 & 0\end{pmatrix}$ & discrete series & fin-dim & $\{\pm \mathrm{Id}\}$\\
         & & of infl char $a$ & of dim $a+1/2$ & \\ \hline
        
          $\begin{pmatrix}a& 0\\0 & -a \end{pmatrix}$ for $a \notin \{\frac{1}{2},\frac{3}{2},...\}$ & $\begin{pmatrix}\exp(\pi i a) & 0\\0 & \exp(-\pi i a)\end{pmatrix}$ & spherical principal series & $\emptyset$ & $H^{\vee}$ or $G^{\vee}$\\ 
          & & of infl char $a$ & & \\ \hline
         
         $\begin{pmatrix}a& 0\\0 & -a \end{pmatrix}$ for $a \notin \{\frac{1}{2},\frac{3}{2},...\}$  & $\begin{pmatrix}-\exp(\pi i a) & 0\\0 & -\exp(-\pi i a)\end{pmatrix}$ & non-spherical principal series & $\emptyset$ & $H^{\vee}$ or $G^{\vee}$\\ 
         & & of infl char $a$ & & \\ \hline
         \end{tabular}
\end{table}
Note: to avoid double-counting, one should assume in the final two rows that $\mathrm{Re}(a) \geq 0$ and $\mathrm{Im}(a) \geq 0$ if $\mathrm{Re}(a)=0$. 
\end{example}

\vspace{5mm}

\subsection{Langlands classification: ABV formulation}\label{sec:ABV}

A fundamental invariant of an admissible representation is its \emph{distribution character}. A natural question is the following:

\begin{question}\label{question1}
Let $\pi \in \Pi(G,\sigma)$. Can we compute the character of $\pi$ from its (complete) Langlands parameter?
\end{question}

In \cite{Langlands}, Langlands introduced the notion of a \emph{standard representation}. This is a finite-length admissible representations which is parabolically induced from a certain (limit of) relative discrete series representation (see \cite[Section 3]{Langlands} or \cite[Chapter 11]{AdamsBarbaschVogan}). Assuming still that $G$ is adjoint, write
\begin{align*}
M(G,\sigma) &:= \{\text{equivalence classes of standard representations of } G^{\sigma}\}\\
M(G,[\sigma]) &:= \bigsqcup_{\sigma' \in [\sigma]} M(G,\sigma')
\end{align*}
Each standard representation has a unique irreducible quotient. This defines a bijection
$$M(G,\sigma) \xrightarrow{\sim} \Pi(G,\sigma).$$
Composing with the bijection of Theorem \ref{thm:LLCclassicalcomplete}, we get a further bijection
$$M: \Phi^c(G^L) \xrightarrow{\sim} M(G,[\sigma])$$
Write
\begin{align*}
K(G,\sigma) &= \text{Grothendieck group of finite-length representations of } G^{\sigma}\\
K(G,[\sigma]) &= \bigoplus_{\sigma' \in [\sigma]} K(G,\sigma')
\end{align*}
Of course, $\Pi(G,[\sigma])$ is a $\ZZ$-basis for $K(G,[\sigma])$. Thus, we can write
$$M(\varphi,\tau) = \sum_{(\varphi',\tau') \in \Phi^c(G^L)} m_r((\varphi,\tau),(\varphi',\tau')) \pi(\varphi',\tau'), \qquad (\varphi,\tau) \in \Phi^c(G^L)$$
Here, $m_r = \{m_r((\varphi,\tau),(\varphi',\tau')) \mid (\varphi,\tau),(\varphi',\tau') \in \Phi^c(G^L)\}$ is an (infinite) square matrix with nonnegative integer entries. In fact, if we choose an appropriate ordering on $\Phi^c(G^L)$, this matrix is upper triangular with ones along the diagonal (see \cite[Proposition 6.6.7]{Vogan1981}). Consequently, $m_r$ is invertible and $M(G,[\sigma])$ forms a second $\ZZ$-basis for $K(G,[\sigma])$. Crucially, the distribution characters of the standard representations $M(G,[\sigma])$ are relatively easy to compute. Thus, Question \ref{question1} becomes.

\begin{question}
Can we compute the (infinite, square) change of basis matrix $m_r$?
\end{question}

A powerful idea, implicit in the Kazhdan-Lusztig conjectures, is that the change-of-basis matrix $m_r$ should be related to the singularities of the $G^{\vee}$-orbit closures in the space $P(G^L)$ of Langlands parameters. Unfortunately, we find that all $G^{\vee}$-orbits in $P(G^L)$ are closed (this is because the image in $G^L$ of a Langlands parameter $\varphi: W_{\RR} \to G^L$ consists of \emph{semisimple} elements). So the $G^{\vee}$-orbit closures in $P(G^L)$ are non-singular. The solution, proposed in \cite{AdamsBarbaschVogan}, is to modify the definition of a Langlands parameter to get a more interesting geometry.

Given a semisimple element $\lambda \in \fg^{\vee}$, define the reductive subalgebra
$$\fm(\lambda) := Z_{\fg^{\vee}}(\exp(2\pi i\lambda)) \qquad \text{(subalgebra corresponding to integral roots)}.$$
There is a $\ZZ$-grading on $\fm(\lambda)$
$$\fm(\lambda) = \bigoplus_{n \in \ZZ}\fm(\lambda)_n, \qquad \fm(\lambda)_n := \{X \in \fm(\lambda) \mid [\lambda,X]=nX\}.$$
We get a parabolic subalgebra
$$\fp(\lambda) = \fl(\lambda) \oplus \fu(\lambda), \qquad \fl(\lambda) = \fm(\lambda)_0, \qquad \fu(\lambda) = \bigoplus_{n >0} \fm(\lambda)_n.$$
The \emph{canonical flat} through $\lambda$ is the affine subspace of $\fm(\lambda)$
$$F(\lambda) := \lambda + \fu(\lambda) = P(\lambda)\lambda \subset \fm(\lambda).$$
Since $\exp(2\pi i \lambda)$ is central in $M^{\vee}$, the function $\exp(2\pi i \bullet)$ is constant on $F(\lambda)$.

\begin{definition}[Definitions 6.9, 7.6, \cite{AdamsBarbaschVogan}]\label{def:geometricparam}
A \emph{geometric parameter} is a pair $(y,F)$ such that
\begin{itemize}
    \item[(i)] $y$ is a semisimple element in $G^L \setminus G^{\vee}$.
    \item[(ii)] $F$ is a canonical flat in $\fg^{\vee}$.
    \item[(iii)] $y^2 = \exp(2\pi i F)$.
\end{itemize}
Write $X(G^L)$ for the set of geometric parameters and $\Xi(G^L)$ for the set of $G^{\vee}$-conjugacy classes in $X(G^L)$. A \emph{complete geometric parameter} is a pair $(\xi,\tau)$ consisting of a geometric parameter $\xi \in X(G^L)$ and an irreducible representation of the component group of $Z_{G^{\vee}}(\xi)$. Write $X^c(G^L)$ for the set of complete geometric parameters and $\Xi^c(G^L)$ for the set of $G^{\vee}$-conjugacy classes in $X^c(G^L)$.
\end{definition}

\begin{prop}[Proposition 6.17, \cite{AdamsBarbaschVogan}]\label{prop:geomofX}
The natural map
$$P(G^L) \twoheadrightarrow X(G^L), \qquad (y,\lambda) \mapsto (y,F(\lambda))$$
induces a bijection
\begin{equation}\label{eq:bij1}\Phi(G^L) \xrightarrow{\sim} \Xi(G^L)\end{equation}
If $(y,\lambda) \in P(G^L)$, then $Z_{G^{\vee}}(y,\lambda)$ is a maximal reductive subgroup of $Z_{G^{\vee}}(y,F(\lambda))$. In particular, there is a canonical isomorphism
$$Z_{G^{\vee}}(y,\lambda)/Z_{G^{\vee}}(y,\lambda)^{\circ} \xrightarrow{\sim} Z_{G^{\vee}}(y,F(\lambda))/Z_{G^{\vee}}(y,F(\lambda))^{\circ} $$
So (\ref{eq:bij1}) lifts to a bijection
\begin{equation}\label{eq:bij2}\Phi^c(G^L) \xrightarrow{\sim} \Xi^c(G^L)\end{equation}
Composing the bijection of Theorem \ref{thm:LLCclassicalcomplete} with (\ref{eq:bij2}), we get a bijection
$$\pi: \Xi^c(G^L) \xrightarrow{\sim} \Pi(G,[\sigma])$$
\end{prop}

In contrast to $P(G^L)$, the set $X(G^L)$ of geometric parameters has an interesting geometry, which we will now describe. There is a decomposition
$$X(G^L) = \bigsqcup_{\OO \subset \fg^{\vee}} X(\OO,G^L), \qquad X(\OO,G^L) = \{(y,F) \in X(G^L) \mid F \subset \OO\}$$
where the disjoint union runs over all semisimple $G^{\vee}$-orbits $\OO \subset \fg^{\vee}$. The pieces $X(\OO,G^L)$ can be built out of the following spaces
\begin{align*}
    F(\OO) &= \{F(\lambda) \mid \lambda \in \OO\}\\
    e(\OO) &= \{\exp(2\pi i \lambda) \mid \lambda \in \OO\}\\
    \sqrt{e(\OO)} &= \{y \in G^L \setminus G^{\vee} \mid y^2 \in e(\OO)\}
\end{align*}
Note that $G^{\vee}$ acts by conjugation on all three spaces above.

The following proposition tells us all that we need to know about the geometry of $X(\OO,G^L)$. The main conclusion is that the $G^{\vee}$-orbit structure on $X(\OO,G^L)$ is completely determined by the orbit structure of a symmetric subgroup acting on a (partial) flag variety. 

\begin{prop}[Proposition 6.16, \cite{AdamsBarbaschVogan}]
The spaces $F(\OO)$, $e(\OO)$, and $\sqrt{e(\OO)}$ are algebraic varieties with $G^{\vee}$-action. The spaces $F(\OO)$ and $e(\OO)$ are homogeneous for $G^{\vee}$. If we fix $\lambda \in \OO$, there are $G^{\vee}$-equivariant isomorphisms
\begin{itemize}
    \item[(i)] $G^{\vee}/P(\lambda) \xrightarrow{\sim} F(\OO)$.
    \item[(ii)] $G^{\vee}/M(\lambda) \xrightarrow{\sim} e(\OO)$.
\end{itemize}
The space $\sqrt{e(\OO)}$ has finitely many $G^{\vee}$-orbits, denoted $\sqrt{e(\OO)}_i$. For each $i$, choose $y_i \in \sqrt{e(\OO)}_i$ with $y_i^2 = e(\lambda)$, and let $K_i(\lambda) = G^{\Ad(y_i)}$, a symmetric subgroup of $M(\lambda)$. There are $G^{\vee}$-equivariant isomorphisms
\begin{itemize}
    \item[(iii)] $G^{\vee}/K_i(\lambda) \xrightarrow{\sim} \sqrt{e(\OO)}_i$.
\end{itemize}
The maps
$$e: F(\OO) \to e(\OO), \qquad (\bullet)^2: \sqrt{e(\OO)} \to e(\OO)$$
are algebraic and $G^{\vee}$-equivariant. Under the identifications of (i), (ii), and (iii), they correspond to the inclusions
$$P(\lambda) \subset M(\lambda), \qquad K_i(\lambda) \subset M(\lambda)$$
respectively. So there is a $G^{\vee}$-equivariant isomorphism
$$X(\OO,G^L) \simeq \bigsqcup_i X(\OO,G^L)_i, \qquad X(\OO,G^L)_i =  G^{\vee}/P(\lambda) \times_{G^{\vee}/M(\lambda)}  G^{\vee}/K_i(\lambda)$$
There are further isomorphisms of $G^{\vee}$-equivariant fiber bundles
$$G^{\vee}/P(\lambda) \times_{G^{\vee}/M(\lambda)}  G^{\vee}/K_i(\lambda) \simeq G^{\vee} \times_{K_i(\lambda)} M(\lambda)/P(\lambda)$$
In particular, there is a natural bijection
$$\{G^{\vee}\text{-orbits on } X(\OO,G^L)_i\} \xrightarrow{\sim} \{K_i(\lambda)\text{-orbits on } M(\lambda)/P(\lambda)\}$$
This bijection preserves closure relations and singularities.
\end{prop}

With $\OO \subset \fg^{\vee}$ fixed, define $\Xi(\OO,G^L)$, $X^c(\OO,G^L)$, $\Xi^c(\OO,G^L)$ by analogy with Definition \ref{def:geometricparam}
\begin{align*}
\Xi(\OO,G^L) &= \{G^{\vee}\text{-conjugacy classes in } X(\OO,G^L)\\
X^c(\OO,G^L) &= \{(\xi,\tau) \in X^c(G^L) \mid \xi \in X(\OO,G^L)\}\\
\Xi^c(\OO,G^L) &= \{G^{\vee}\text{-conjugacy classes in } X^c(\OO,G^L)\}
\end{align*}
Also define
$$KX(\OO,G^L) = \text{Grothendieck group of } G^{\vee}\text{-equivariant constructible sheaves on } X(\OO,G^L)$$
A conjugacy class $G^{\vee}(\xi,\tau) \in \Xi^c(\OO,G^L)$ is equivalent to a pair $(S_{\xi},\mathcal{L}_{\tau})$ consisting of a $G^{\vee}$-orbit $S_{\xi} = G^{\vee}\xi$ in $X(\OO,G^L)$ and a $G^{\vee}$-equivariant local system $\mathcal{L}_{\tau}$ on $S_{\xi}$. Given such a pair, we can construct two different classes $\mu(\xi,\tau)$ and $P(\xi,\tau)$ in $KX(\OO,G^L)$. Let $j: S_{\xi} \hookrightarrow X(\OO,G^L)$ denote the (locally closed) embedding. For $\mu(\xi,\tau)$, we fist regard $\mathcal{L}_{\tau}$ as a simple $G^{\vee}$-equivariant constructible sheaf on $S_{\xi}$. Taking the pushforward with proper support, we get a simple $G^{\vee}$-equivariant constructible sheaf $j_!\mathcal{L}_{\tau}$ on $X(\OO,G^L)$. We define $\mu(\xi,\tau)$ to be the corresponding class in $K$-theory (with a sign determined by the parity $\dim(S_{\xi})$)
$$\mu(\xi,\tau) :=  (-1)^{\dim(S_{\xi})}[j_! \mathcal{L}_{\tau}] \in KX(\OO,G^L)$$
For $P(\xi,\tau)$, we start with the complex $\mathcal{L}_{\tau}[\dim(S_{\xi})]$ consisting of a single sheaf $\mathcal{L}_{\tau}$ in degree $-\dim(S_{\xi})$. This complex is a simple $G^{\vee}$-equivariant perverse sheaf on $S_{\xi}$. Taking the intermediate extension, we get a simple $G^{\vee}$-equivariant perverse sheaf $j_{!*}\mathcal{L}_{\tau}$ on $X(\OO,G^L)$. By definition, $j_{!*}\mathcal{L}_{\tau}$ is a bounded complex of sheaves on $X(\OO,G^L)$ with constructible cohomology. We define $P(\xi,\tau)$ to be the Euler characterstic of this complex, i.e. the alternating sum of its cohomology sheaves
\begin{equation}\label{eq:perverse}P(\xi,\tau) := \sum_i (-1)^i [H^i(j_{!*}\mathcal{L}_{\tau})] \in KX(\OO,G^L).\end{equation}
The sets $\{\mu(\xi,\tau) \mid (\xi,\tau) \in \Xi^c(\OO,G^L)\}$ and $\{P(\xi,\tau) \mid (\xi,\tau) \in \Xi^c(\OO,G^L)\}$ are $\ZZ$-bases for $KX(\OO,G^L)$, see \cite[Theorem 4.3.1]{BeilinsonBernsteinDeligne}. Now let
$$KX(G^L) := \bigoplus_{\OO \subset \fg^{\vee}} KX(G^L,\OO),$$
where the sum runs over all semisimple orbits $\OO \subset \fg^{\vee}$. Of course, the sets $\{\mu(\xi,\tau) \mid (\xi,\tau) \in \Xi^c(G^L)\}$ and $\{P(\xi,\tau) \mid (\xi,\tau) \in \Xi^c(G^L)\}$ are $\ZZ$-bases for $KX(G^L)$. Write $m_g$ for the change of basis matrix
$$\mu(\xi,\tau) = \sum_{(\xi',\tau') \in \Xi^c(G^L)} m_g((\xi,\tau),(\xi',\tau'))P(\xi',\tau'), \qquad (\xi,\tau) \in \Xi^c(G^L)$$
Thanks to Proposition \ref{prop:geomofX}, the problem of computing $m_g$ can be reduced to an analogous problem for a symmetric subgroup acting on a flag variety. This analogous problem was solved in \cite{LusztigVogan}. The conclusion is that the matrices $m_r$ and $m_g$ are (essentially) mutually inverse transpose.

\begin{theorem}[Corollary 1.25, \cite{AdamsBarbaschVogan}]\label{thm:duality1}
The matrices $m_r$ and $m_g$ are related by the formula
$$m_r((\xi,\tau),(\xi',\tau')) = (-1)^{\dim(S_{\xi})-\dim(S_{\xi'})} m_g^{-1}((\xi',\tau'),(\xi,\tau))$$
\end{theorem}

It is conceptually illuminating to reformulate Theorem \ref{thm:duality1} as a duality on the level of Grothendieck groups.

\begin{cor}[Theorem 1.24, \cite{AdamsBarbaschVogan}]
There is a perfect pairing
$$\langle \ , \ \rangle: K\Pi(G,[\sigma]) \times KX(G^L) \to \ZZ$$
such that
\begin{itemize}
    \item[(i)] $\langle M(\xi,\tau),\mu(\xi',\tau')\rangle = e(\xi,\tau) \delta_{(\xi,\tau),(\xi',\tau')}$.
    \item[(ii)] $\langle \pi(\xi,\tau), P(\xi',\tau')\rangle = e(\xi,\tau) (-1)^{\dim(S_{\xi})}\delta_{(\xi,\tau),(\xi',\tau')}$
\end{itemize}
where $e(\xi,\tau) \in \{\pm 1\}$ is the \emph{Kottwitz sign} (cf. \cite[Definition 15.8]{AdamsBarbaschVogan}) and $\delta$ is the Kronecker delta.
\end{cor}

\begin{example}\label{ex:PGL3}
Let $G=\mathrm{PGL}(2,\CC)$ and let $[\sigma]$ be the unique inner class. We will use the notation of Examples \ref{ex:PGL1} and \ref{ex:PGL2}. Let 
$$\lambda = \rho = \begin{pmatrix} \frac{1}{2} & 0\\0 & -\frac{1}{2} \end{pmatrix}, \qquad \OO = G^{\vee}\lambda$$
We will describe the space $X(G^L,\OO)$ using Proposition \ref{prop:geomofX}. First note that $M(\lambda)=G^{\vee}$ (since $\lambda$ is integral) and $e(\OO) = \{-\mathrm{Id}\}$. Since $\lambda$ is regular, $P(\lambda) = B^{\vee}$, the subgroup of $G^{\vee}$ consisting of upper triangular matrices, and $F(\OO) \simeq G^{\vee}/B^{\vee}$. Finally
$$\sqrt{e(\OO)} = \{ y \in G^L \setminus G^{\vee} \mid y^2 = -\mathrm{Id}\} = \Ad(G^{\vee}) (\begin{pmatrix}i & 0\\0 & -i \end{pmatrix}, \delta) \simeq G^{\vee}/T^{\vee},$$
where $T^{\vee} = K_1(\lambda)$ is the diagonal torus, the fixed points of $\Ad\begin{pmatrix}i & 0\\0 & -i\end{pmatrix}$. So by Proposition \ref{prop:geomofX}
$$X(\OO,G^L) = G^{\vee}/B^{\vee} \times G^{\vee}/T^{\vee} \simeq G^{\vee} \times_{T^{\vee}} G^{\vee}/B^{\vee}$$
and the $G^{\vee}$-orbit structure on $X(\OO,G^L)$ mirrors the $T^{\vee}$-orbit structure on $G^{\vee}/B^{\vee}$. Accordingly, we can identify the Grothendieck group $KX(G^L,\OO)$ with the Grothendieck group of $T^{\vee}$-equivariant constructible sheaves on $G^{\vee}/B^{\vee}$ (and $m_g$ is preserved under this identification). 

Note that $G^{\vee}/B^{\vee} \simeq \mathbb{P}^1(\CC)$ and $T^{\vee}$ acts therein with three orbits: the north pole $\{N\}$, the south pole $\{S\}$, and open complement $U = \mathbb{P}^1(\CC) \setminus \{N,S\} \simeq T^{\vee}/\{\pm \mathrm{Id}\}$. Now we can describe the set $\Xi^c(\OO,G^L)$ of complete geometric parameters $(S, \mathcal{L})$:
$$\Xi^c(\OO,G^L) = \{(\{N\},\mathcal{L}^N_{\mathrm{triv}}), (\{S\},\mathcal{L}^S_{\mathrm{triv}}), (U,\mathcal{L}^U_{\mathrm{triv}}), (U,\mathcal{L}^U_{\mathrm{sgn}})\}$$
Here $\mathcal{L}^U_{\mathrm{triv}}$ denotes the $T^{\vee}$-equivariant local system on $U$ corresponding to the trivial character of $\{\pm \mathrm{Id}\}$ and so on. Each parameter $(S,\mathcal{L})$ corresponds to a standard $M(S,\mathcal{L})$ and irreducible $\pi(S,\mathcal{L})$ representation of a real form of $G$ as well as a constructible $\mu(S,\mathcal{L})$ and perverse $P(S,\mathcal{L})$ sheaf on $X(\OO,G^L)$. We will record these correspondences in the table below.
\begin{center}
\begin{tabular}{|c|c|c|c|c|} \hline
   $(S,\mathcal{L})$  & $\pi(S,\mathcal{L})$ & $M(S,\mathcal{L})$ & $\mu(S,\mathcal{L})$ & $P(S,\mathcal{L})$\\ \hline
   $(N,\mathcal{L}^N_{\mathrm{triv}})$ & trivial rep of $G_s$ & spherical principal series of $G_s$ & $\delta_N$ & $\delta_N$ \\ \hline
   $(S,\mathcal{L}^S_{\mathrm{triv}})$ & sign rep of $G_s$ & non-spherical principal series of $G_s$ &  $\delta_S$ & $\delta_S$\\ \hline
   
   $(U,\mathcal{L}^U_{\mathrm{triv}})$ & discrete series of $G_s$ & discrete series of $G_s$ &  $j_!\mathcal{L}^U_{\mathrm{triv}}[1]$ & $\mathcal{L}_{\mathrm{triv}}^{\mathbb{P}^1}[1]$
    \\ \hline
   $(U,\mathcal{L}^U_{\mathrm{sgn}})$ & triv of $G_c$ & triv of $G_c$ & $j_!\mathcal{L}^U_{\mathrm{sgn}}[1]$ & $j_!\mathcal{L}^U_{\mathrm{sgn}}[1]$ \\ \hline
\end{tabular}
\end{center}
Here $\delta_N$, $\delta_S$ are the skyscraper sheaves at $N$, $S$ and $\mathcal{L}^{\mathbb{P}^1}_{\mathrm{triv}}$ is the constant sheaf on $\mathbb{P}^1(\CC)$.

The matrix $m_r$ is
$$m_r = \begin{pmatrix}
1 & 0 & 1 & 0\\
0 & 1 & 1 & 0\\
0 & 0 & 1 & 0\\
0 & 0 & 0 & 1
\end{pmatrix} 
$$
(here we use the following easy facts: the spherical (resp. non-spherical) principal series representation of $G_s$ is the sum in $K\Pi(G,\sigma_s)$ of the trivial (resp. sign) representation and the discrete series representation of infinitesimal character $\rho$).

The matrix $m_g$ is: 
$$m_g = \begin{pmatrix}
1 & 0 & 0 & 0\\
0 & 1 & 0 & 0\\
1 & 1 & 1 & 0\\
0 & 0 & 0 & 1
\end{pmatrix} $$
(here we use the following fact, which follows from an easy distinguished triangle: $ j_!\mathcal{L}^U_{\mathrm{triv}}[1] = \mathcal{L}^{\mathbb{P}^1}_{\mathrm{triv}}[1] + \delta_N + \delta_S$ in the Grothendieck group of $T^{\vee}$-equivariant constructible sheaves on $\mathbb{P}^1(\CC)$). It is easy to check that $m_r$ and $m_g$ satisfy the relation of Theorem \ref{thm:duality1}. 
\end{example}

\section{Arthur's conjectures}\label{sec:Arthur}

\begin{definition}[p. 10, \cite{Arthur1983}]
An Arthur parameter for $G^L$ is a continuous homomorphism 
$$\psi: W_{\RR} \times \mathrm{SL}(2,\CC) \to G^L$$
such that
\begin{itemize}
    \item[(i)] $\psi|_{W_{\RR}}$ is a tempered Langlands parameter (cf. Definition \ref{def:Lparameter}).
    \item[(ii)] the restriction of $\psi$ to $\mathrm{SL}(2,\CC)$ is algebraic.
\end{itemize}
Write $Q(G^L)$ for the set of Arthur parameters for $G^L$ and $\Psi(G^L)$ for the set of $G^{\vee}$-conjugacy classes in $Q(G^L)$. An Arthur parameter is \emph{unipotent} if $\psi$ is trivial on $\CC^{\times} \subset W_{\RR}$. Write $Q_{\mathrm{unip}}(G^L)$ for the set of unipotent Arthur parameters and $\Psi_{\mathrm{unip}}(G^L) \subset \Psi(G^L)$ for the set of $G^{\vee}$-conjugacy classes in $Q_{\mathrm{unip}}(G^L)$.
\end{definition}

For each Arthur parameter $\psi \in Q(G^L)$, there is an associated Langlands parameter $\phi_{\psi} \in P(G^L)$ defined by the formula
$$\phi_{\psi}: W_{\RR} \to G^L, \qquad \phi_{\psi}(w) = \psi(w,\begin{pmatrix}|w|^{1/2} & 0\\0 & |w|^{-1/2} \end{pmatrix})$$
This defines a $G^{\vee}$-equivariant map $Q(G^L) \to P(G^L)$. The induced map on conjugacy classes
\begin{equation}\label{eq:injection0}\Psi(G^L) \hookrightarrow \Phi(G^L)\end{equation}
is injective (see \cite[Proposition 1.3.1]{Arthur1983}). Write $\Phi_{\mathsf{Art}}(G^L)$ for its image. If $\psi$ is trivial on $\mathrm{SL}(2,\CC)$, then $\phi_{\psi} = \psi|_{W_{\RR}}$, a tempered Langlands parameter. Hence we have inclusions
$$\Phi_{\mathrm{temp}}(G^L) \subset \Phi_{\mathrm{Art}}(G^L) \subset \Phi(G^L)$$
Composing the map (\ref{eq:injection0}) with the bijection $\Phi(G^L) \xrightarrow{\sim} \Xi(G^L)$ (\ref{eq:bij1}), we get an injection
\begin{equation}\label{eq:injection}
\Psi(G^L) \hookrightarrow \Xi(G^L)\end{equation}
Write $\Xi_{\mathsf{Art}}(G^L) \subset \Xi(G^L)$ for its image. The local version (over $\RR$) of Arthur's conjectures can be stated succinctly as follows.

\begin{conj}[Conjecture 1.3.2, \cite{Arthur1983}]\label{conj:Arthur}
For each Arthur parameter $\psi \in \Psi(G^L)$, there is an associated finite set of irreducible admissible representations of $G^{\sigma}$
$$\Pi_{\psi}^{\mathsf{Art}}(G,\sigma) \subset \Pi(G,\sigma)$$
These sets should satisfy several properties, including
\begin{itemize}
    \item[(i)] $\Pi_{\phi_{\psi}}(G,\sigma) \subset \Pi^{\mathsf{Art}}_{\psi}(G,\sigma)$.
    \item[(ii)] The constituents of $\Pi^{\mathsf{Art}}_{\psi}(G,\sigma)$ are unitary.
\end{itemize}
\end{conj}

The packets $\Pi^{\mathsf{Art}}_{\psi}(G,\sigma)$ are also conjectured to possess an array of additional properties related to endoscopy and stability (see \cite[Chapter 1]{AdamsBarbaschVogan} for a detailed discussion). These additional properties are extremely important, but I will not discuss them here. Another desideratum, which is implicit in \cite{Arthur1983}, is that the Arthur packets should exhaust most (but not all) of $\Pi_u(G,\sigma)$. I will not attempt to make this more precise, but we will get some flavor of what it might mean in Example \ref{ex:PGL3}.

The packets $\Pi_{\psi}^{\mathsf{Art}}(G,\sigma)$ were defined in \cite{AdamsBarbaschVogan} in terms of microlocal geometry on the geometric parameter space $X(G^L)$. I will recall their definition here. For simplicity of exposition, we continue to assume that $G$ is adjoint. 

If $M$ is a $G^{\vee}$-equivariant perverse sheaf on $X(\OO,G^L)$, there is a \emph{characteristic cycle}
$$\chi(M) = \sum_{S \subset X(\OO,G^L)} \chi_S(M)\overline{T^*_SX(\OO,G^L)}, \qquad \chi_S(M) \in \ZZ_{\geq 0}$$
Here $T^*_SX(\OO,G^L) \subset T^*X(\OO,G^L)$ is the conormal bundle for $S \subset X(\OO,G^L)$ and the sum runs over all $G^{\vee}$-orbits $S$ in $X(\OO,G^L)$. This invariant is additive on short exact sequences, i.e. the functions $\chi_S$ can be regarded as $\ZZ$-linear functionals on the Grothendieck group $KX(\OO,G^L)$. 

\begin{definition}[Definition 19.15, \cite{AdamsBarbaschVogan}]\label{def:Arthur}
Let $\psi \in \Psi(G^L)$. Let $S \in \Xi_{\mathsf{Art}}(G^L)$ be the image of $\psi$ under the embedding (\ref{eq:injection}). Suppose $S \subset X(\OO,G^L)$ for a semisimple $G^{\vee}$-orbit $\OO \subset \fg^{\vee}$. Define
$$\Xi_{\psi}^{\mathsf{Art}}(G^L) := \{(\xi,\tau) \in \Xi^c(\OO,G^L) \mid \chi_S(P(\xi,\tau)) \neq 0\}$$
where $P(\xi,\tau) \in KX(\OO,G^L)$ is the class defined in (\ref{eq:perverse}). Then we define
$$\Pi_{\psi}^{\mathsf{Art}}(G, [\sigma]) := \{\pi(\xi,\tau) \mid (\xi,\tau) \in \Xi_{\psi}^{\mathsf{Art}}(\OO,G^L)\}, \qquad \Pi_{\psi}^{\mathsf{Art}}(G,\sigma) := \Pi_{\psi}^{\mathsf{Art}}(G,[\sigma]) \cap \Pi(G,\sigma)$$
\end{definition}

It is easy to see that the packets $\Pi_{\psi}^{\mathsf{Art}}(G,\sigma)$ of Definition \ref{def:Arthur} satisfy Condition (i) of Conjecture \ref{conj:Arthur} (this is a consequence of the following easy fact: the characteristic cycle of the perverse sheaf $P(\xi,\tau)$ includes the conormal bundle for its support $G^{\vee}\xi \subset X(\OO,G^L)$). Also established in \cite{AdamsBarbaschVogan} are the other properties alluded to after Conjecture \ref{conj:Arthur}. Notably, condition (ii) of Conjecture \ref{conj:Arthur} is \emph{not} proved in \cite{AdamsBarbaschVogan}. The unitarity of Arthur packets remains an open problem in general and is an active area of research.

\begin{example}
Let $G=\mathrm{PGL}(2,\CC)$. Fix the notation of Examples \ref{ex:PGL1}, \ref{ex:PGL2}, and \ref{ex:PGL3}. We will compute the set $\Psi(G^L)$ of Arthur parameters for $G^L$ and describe the corresponding Arthur packets. 

Up to conjugation by $G^{\vee}$, there are two algebraic homomorphisms $\mathrm{SL}(2,\CC) \to G^L$: the trivial map and the natural inclusion $\mathrm{SL}(2,\CC) = G^{\vee} \subset G^L$. If $\psi|_{\mathrm{SL}(2,\CC)}$ is the inclusion, there are two possibilities for $\psi$: either $\psi|_{W_{\RR}}$ is trivial or $\psi(W_{\RR}) = \Gamma \subset G^L$. Denote the corresponding (unipotent) Arthur parameters by $\psi^+$ and $\psi^-$, respectively. 

If $\psi|_{\mathrm{SL}(2,\CC)}$ is trivial, then $\psi|_{W_{\RR}}$ can be an arbitrary tempered Langlands parameter. According to Proposition \ref{prop:propsofylambda} and Example \ref{ex:PGL1}, the tempered Langlands parameters, up to conjugation by $G^{\vee}$, are given by the pairs $(y,\lambda)$ of the form
\begin{itemize}
    \item[(i)] $\lambda = \begin{pmatrix}a & 0\\0 & -a\end{pmatrix}$ and $ y=\begin{pmatrix}0 & 1\\-1 & 0\end{pmatrix}$ for $a = \frac{1}{2},\frac{3}{2}, ...$, or
    \item[(ii)] $\lambda = \begin{pmatrix}a & 0\\0 & -a\end{pmatrix}$ and $y =  \begin{pmatrix}\mathrm{exp}(\pi i a) & 0\\0 & \mathrm{exp}(-\pi i a)\end{pmatrix}$ for $a \in i\RR_{\geq 0}$, or
    \item[(iii)] $\lambda = \begin{pmatrix}a & 0\\0 & -a\end{pmatrix}$ and $y =  \begin{pmatrix}-\mathrm{exp}(\pi i a) & 0\\0 & -\mathrm{exp}(-\pi i a)\end{pmatrix}$ for $a \in i\RR_{\geq 0}$.
\end{itemize}
Denote the corresponding Arthur parameters by
\begin{itemize}
    \item[(i)] $\psi_a^{DS}$ for $a = \frac{1}{2},\frac{3}{2},...$ (`$DS$' stands for `discrete series').
    \item[(ii)] $\psi_a^{SP}$ for $a \in i\RR_{\geq 0}$ (`$SP$' stands for `spherical principal series')
    \item[(iii)] $\psi_a^{NSP}$ for $a \in i\RR_{\geq 0}$ (`$NSP$' stands for `non-spherical principal series')
\end{itemize}
Thus, we have determined
$$\Psi(G^L) = \{\psi^+,\psi^-\} \cup \{\psi^{DS}_a \mid a = \frac{1}{2},\frac{3}{2},...\} \cup \{\psi_a^{SP}, \psi_a^{NSP} \mid a \in i\RR_{\geq 0}\}.$$
Below, we describe the corresponding Langlands parameters $\phi_{\psi}$ (in terms of the pairs $(\lambda,y)$ of Proposition \ref{prop:params}) as well as the associated $L$-packets
\begin{table}[H]
\tiny
    \begin{tabular}{|c|c|c|c|c|} \hline
        $\psi$ &  $\lambda$ & $y$ & $\Pi_{(\lambda,y)}(G, \sigma_s)$ & $\Pi_{(\lambda,y)}(G, \sigma_c)$\\ \hline
        $\psi^+$ & $\begin{pmatrix}\frac{1}{2} & 0\\0 & -\frac{1}{2}\end{pmatrix}$ & $\begin{pmatrix}i &0\\0 & -i\end{pmatrix}$ & $\{\text{triv}\}$ & $\emptyset$\\ \hline
        $\psi^-$ & $\begin{pmatrix}\frac{1}{2} & 0\\0 & -\frac{1}{2}\end{pmatrix}$ & $\begin{pmatrix}-i &0\\0 & i\end{pmatrix}$ & $\{\text{sgn}\}$ & $\emptyset$\\ \hline
        $\psi^{DS}_a$ & $\begin{pmatrix}a & 0\\0 & -a\end{pmatrix}$ & $\begin{pmatrix}0 & 1\\-1 & 0\end{pmatrix}$ & $\{\text{discrete series of infl char }a\}$ & $\{\text{fin-dim of dim }a+\frac{1}{2}\}$\\ \hline
        
        $\psi_a^{SP}$ & $\begin{pmatrix}a & 0\\0 & -a\end{pmatrix}$ & $\begin{pmatrix} \exp(\pi i a) & 0 \\ 0 & \exp(-\pi i a) \end{pmatrix}$ & $\{\text{sph prin series of infl char } a\}$ & $\emptyset$\\ \hline 
        $\psi_a^{NSP}$ & $\begin{pmatrix}a & 0\\0 & -a\end{pmatrix}$ & $\begin{pmatrix} -\exp(\pi i a) & 0 \\ 0 & -\exp(-\pi i a) \end{pmatrix}$ & $\{\text{nonsph prin series of infl char } a\}$ & $\emptyset$\\ \hline 
    \end{tabular}
\end{table}
The tempered Langlands parameters
$$\{\phi_{\psi^{DS}_a} \mid a = \frac{1}{2},\frac{3}{2},...\} \cup \{\phi_{\psi^{SP}_a}, \phi_{\psi^{NSP}_a} \mid a \in i\RR_{\geq 0}\}$$
correspond under the bijection $\Phi(G^L) \xrightarrow{\sim} \Xi(G^L)$ to open $G^{\vee}$-orbits in varieties of the form $X(\OO,G^L)$. Thus, the Arthur packet $\Pi^{\mathsf{Art}}_{\psi}(G,[\sigma])$ coincides with the tempered $L$-packet $\Pi_{\phi_{\psi}}(G,[\sigma])$ for $\psi \in \{\psi_a^{DS}\} \cup \{\psi_a^{SP}, \psi_a^{NSP}\}$.

It remains to compute the unipotent Arthur packets $\Pi^{\mathsf{Art}}_{\psi^{\pm}}(G,[\sigma])$. Note that the Langlands parameters $\phi_{\psi^+}$ and $\phi_{\psi^-}$ correspond to the closed $G^{\vee}$-orbits $\{N\}$ and $\{S\}$, respectively, in the variety $X(\OO,G^L)$ for $\OO = G^{\vee}\rho \subset \fg^{\vee}$ (see Example \ref{ex:PGL3}). To compute the sets $\Xi^{\mathsf{Art}}_{\psi^{\pm}}(G^L)$, we must compute the characteristic cycles of the simple perverse sheaves $\mathcal{L}^{\mathbb{P}^1}_{\mathrm{triv}}[1]$ and $j_!\mathcal{L}_{\mathrm{sgn}}^U[1]$ on $\mathbb{P}^1(\CC)$ (see again Example \ref{ex:PGL3} for notation). If $X$ is a smooth irreducible curve and $P$ is a perverse sheaf on $X$, then it is easy to compute the characteristic cycle $\chi(P)$ of $P$. For a point $x \in X$, write $n_x$ for the Euler characteristic of the stalk of $P$ at $x$. Then $n_x$ takes a constant value $n$ on a Zariski-open subset $X \setminus \{x_1,...,x_m\}$ and $n_{x_i} > n$ for $i=1,...,m$. In this case
$$\chi(P) = -nX -\sum_{i=1}^m (n-n_{x_i})T_{x_i}^*X$$
see, for example, \cite[Example 4.3.21(iii)]{Dimca}.
So in our case, we obtain
\begin{align*}
\chi(\mathcal{L}^{\mathbb{P}^1}_{\mathrm{triv}}[1]) &= X(\OO,G^L)\\
\chi(j_!\mathcal{L}_{\mathrm{sgn}}^U[1]) &= X(\OO,G^L) + T^*_{N}X(\OO,G^L) + T^*_SX(\OO,G^L)\end{align*}
Hence, 
$$\Xi_{\psi^+}^{\mathsf{Art}}(G^L) = \{(\{N\},\mathcal{L}^N_{\mathrm{triv}}), (U,\mathcal{L}_{\mathrm{sgn}}^U)\}, \qquad \Xi_{\psi^-}^{\mathsf{Art}}(G^L) = \{(\{S\},\mathcal{L}^S_{\mathrm{triv}}), (U,\mathcal{L}_{\mathrm{sgn}}^U)\}$$
So
$$\Pi^{\mathsf{Art}}_{\psi^+}(G, [\sigma]) = \{\text{triv of } G_s, \text{triv of } G_c\}, \qquad \Pi^{\mathsf{Art}}_{\psi^-}(G, [\sigma]) = \{\text{sgn of } G_s, \text{triv of } G_c\}$$
\begin{table}[H]
\tiny
    \begin{tabular}{|c|c|c|c|c|} \hline
        $\psi$ &  $\lambda$ & $y$ & $\Pi^{\mathsf{Art}}_{\psi}(G,\sigma_s)$ & $\Pi^{\mathsf{Art}}_{\psi}(G,\sigma_c)$\\ \hline
        $\psi^+$ & $\begin{pmatrix}\frac{1}{2} & 0\\0 & -\frac{1}{2}\end{pmatrix}$ & $\begin{pmatrix}i &0\\0 & -i\end{pmatrix}$ & $\{\text{triv}\}$ & $\{\text{triv}\}$\\ \hline
        $\psi^-$ & $\begin{pmatrix}\frac{1}{2} & 0\\0 & -\frac{1}{2}\end{pmatrix}$ & $\begin{pmatrix}-i &0\\0 & i\end{pmatrix}$ & $\{\text{sgn}\}$ & $\{\text{triv}\}$\\ \hline
        $\psi_a^{DS}$ & $\begin{pmatrix}a & 0\\0 & -a\end{pmatrix}$ & $\begin{pmatrix}0 & 1\\-1 & 0\end{pmatrix}$ & $\{\text{discrete series of infl char }a\}$ & $\{\text{fin-dim of dim }a+\frac{1}{2}\}$\\ \hline
        
        $\psi_a^{SP}$ & $\begin{pmatrix}a & 0\\0 & -a\end{pmatrix}$ & $\begin{pmatrix} \exp(\pi i a) & 0 \\ 0 & \exp(-\pi i a) \end{pmatrix}$ & $\{\text{sph prin series of infl char } a\}$ & $\emptyset$\\ \hline 
        $\psi_a^{NSP}$ & $\begin{pmatrix}a & 0\\0 & -a\end{pmatrix}$ & $\begin{pmatrix} -\exp(\pi i a) & 0 \\ 0 & -\exp(-\pi i a) \end{pmatrix}$ & $\{\text{nonsph prin series of infl char } a\}$ & $\emptyset$\\ \hline 
    \end{tabular}
\end{table}
All of the representations in the table above are unitary. However, there are some unitary representations of $G_s$ which are \emph{not} contained in an Arthur packet, namely the unitary complementary series. This seems to be a general phenomenon.
\end{example}

We note that for quasi-split classical groups, Arthur in (\cite{Arthur2013}) has defined the packets $\Pi_{\psi}^{\mathsf{Art}}(G,\sigma)$ using rather different methods (his construction works over any local field). He proves that his packets satisfy all of the conjectured properties, including unitarity \cite[Theorems 1.5.1 and 2.2.1]{Arthur2013}. It is natural to ask whether the packets defined by Arthur coincide with those of Definition \ref{def:Arthur}. It was proved that they do in \cite{equivalentdefinitions} (modulo a small technical issue for even rank orthogonal groups).

\section{The Orbit Method}\label{sec:orbit}

Let $G_{\RR}$ be the real points of a connected reductive algebraic group defined over $\RR$ (since we will be working in this section with a \emph{single} real form of $G$, we will drop the symbol `$\sigma$' from all of our notation). Let $\mathfrak{g}_{\RR}$ be the Lie algebra of $G_{\RR}$. Note that $G_{\RR}$ acts on the real vector space $\Hom_{\RR}(\fg_{\RR},i\RR)$. By a \emph{real co-adjoint orbit} we will mean a $G_{\RR}$-orbit $\OO$ on $\Hom_{\RR}(\fg_{\RR},i\RR)$ and by a \emph{real co-adjoint cover} we will mean a homogeneous $G_{\RR}$-space $\widetilde{\OO}$ equipped with a finite $G_{\RR}$-equivariant map $\widetilde{\OO} \to \OO$ to a real co-adjoint orbit. Write $\mathsf{Orb}^{i\RR}(G_{\RR})$ (resp. $\mathsf{Cov}^{i\RR}(G_{\RR})$) for the set of  real co-adjoint orbits (resp. isomorphism classes of real co-adjoint covers). The \emph{orbit method} is a guiding philosophy in unitary representation theory, originating in the work of Kirillov (\cite{Kirillov1962}) and Kostant (\cite{Kostant1970}). Here is a (somewhat precise) formulation due to Vogan (\cite{Voganorbit}).

\begin{conj}\label{conj:orbit}
To each real co-adjoint cover $\widetilde{\OO} \in \mathsf{Cov}^{i\RR}(G_{\RR})$ (satisfying a suitable integrality condition) there is an associated finite set of irreducible unitary representations
$$\Pi^{\mathsf{Kir}}_{\widetilde{\OO}}(G_{\RR}) \subset \Pi_u(G_{\RR}).$$
The union 
    $$\bigcup_{\widetilde{\OO}} \Pi^{\mathsf{Kir}}_{\widetilde{\OO}}(G_{\RR})$$
should exhaust most of $\Pi_u(G_{\RR})$. 
\end{conj}
There are several problems with Conjecture \ref{conj:orbit} (they are problems with our simplified formulation, not with the discussion in \cite{Voganorbit}):
\begin{enumerate}
    \item The `integrality condition' and the phrase `most of $\Pi_u(G_{\RR})$' have not been precisely defined.
    \item In some cases, we also expect packets attached to \emph{unions} of co-adjoint covers. 
\end{enumerate}

We will return to Problem 1 in Section \ref{sec:integrality}. Problem 2 will not concern us (it does not arise in the case of complex groups, which is the main case will consider).

Below, we will define the orbit method in the case of \emph{complex groups}. So let $G$ be a complex connected reductive algebraic group, regarded as a real Lie group by restriction of scalars. Of course, $G$ is the real points of a complex connected reductive algebraic group (isomorphic to the product $G \times G$), but to simplify the notation we will write $\Pi(G)$ (resp. $\Pi_u(G)$, resp. $\Pi^{\mathsf{Art}}_{\psi}(G)$) instead of $\Pi(G \times G,\sigma)$ (resp. $\Pi_u(G \times G,\sigma)$, resp. $\Pi^{\mathsf{Art}}_{\psi}(G \times G,\sigma)$). 

Let $\fg^*$ denote the \emph{complex} dual space, i.e. $\fg^* := \Hom_{\CC}(\fg,\CC)$. By a \emph{complex co-adjoint orbit} we will mean  a $G$-orbit $\OO$ on $\fg^*$ and by a \emph{complex co-adjoint cover} we will mean a homogeneous $G$-space $\widetilde{\OO}$ equipped with a finite $G$-equovariant map $\widetilde{\OO} \to \OO$ to a complex co-adjoint orbit. Write $\mathsf{Orb}(G)$ (resp. $\mathsf{Cov}(G)$) for the set of isomorphism classes of complex co-adjoint orbits (resp. co-adjoint covers). A co-adjoint cover is \emph{nilpotent} if the underlying orbit is nilpotent. To define our orbit method, it will be convenient to relate real and complex co-adjoint covers. This is accomplished using the following easy lemma. The trivial proof is omitted.

\begin{lemma}\label{lem:imaginary}
If $\mu \in \fg^*$, define $\iota(\mu) \in \Hom_{\RR}(\fg,i\RR)$ by
$$(\iota(\mu))(X) = \mathrm{Im}(\mu(X)).$$
where $\mathrm{Im}(\bullet)$ denotes the imaginary part of a complex number. Then $\mu \mapsto \iota(\mu)$ defines a $G$-equivariant isomorphism of real vector spaces
\begin{equation}\label{eq:iota}\iota: \fg^* \xrightarrow{\sim} \Hom_{\RR}(\fg,i\RR).\end{equation}
The inverse isomorphism is defined by the formula
$$\iota^{-1}(\lambda)(X) = \lambda(X) - i\lambda(iX).$$
The isomorphism $\iota$ induces a bijection (also denoted by $\iota$)
\begin{equation}\label{eq:iotaorb}\iota: \mathsf{Orb}(G) \xrightarrow{\sim} \mathsf{Orb}^{i\RR}(G)\end{equation}
If $\mu \in \fg^*$, then the centralizers $G^{\mu}$ and $G^{\iota(\mu)}$ coincide. So (\ref{eq:iotaorb}) lifts to a bijection (still denoted by $\iota$)
\begin{equation}\label{eq:iotacovers}\iota: \mathsf{Cov}(G) \xrightarrow{\sim} \mathsf{Cov}^{i\RR}(G)\end{equation}
\end{lemma}

\subsection{Birational induction}\label{sec:bind}

\begin{definition}[Section 7.8, \cite{LMBM}, see also Definition 1.2, \cite{Losev4}]\label{def:inductiondatum}
A \emph{birational induction datum} is a triple
$$(L,\widetilde{\OO}_L,\xi)$$
consisting of
\begin{itemize}
    \item A Levi subgroup $L \subset G$.
    \item A complex nilpotent cover $\widetilde{\OO}_L$.
    \item An element $\xi \in \fz(\fl)^*$.
\end{itemize}
Note that $G$ acts by conjugation on the set of birational induction data. Denote the set of $G$-conjugacy classes by $\Omega(G)$.
\end{definition}

We will define a correspondence called \emph{birational induction} from birational induction data to complex co-adjoint covers. This map is essentially defined in \cite[Section 4]{Losev4}.

First, choose a parabolic subgroup $P = LN \subset G$ and write $\fp^{\perp}$ for the annihilator of $\fp$ in $\fg^*$. Consider the twisted generalized Springer map
$$\mu: G \times_P (\xi + \overline{\OO}_L + \fp^{\perp}) \to \fg^*$$
Here, we view $\xi+\overline{\OO}_L + \fp^{\perp}$ as a (closed, $P$-invariant) subset of $\fn^{\perp}$. The image of $\mu$ is the closure of a single co-adjoint $G$-orbit, denoted $\Ind (L,\OO_L,\xi) \subset \fg^*$ (if $\xi=0$, then $\Ind(L,\OO_L,\xi)$ is the induced nilpotent orbit in the sense of \cite{LusztigSpaltenstein}).

Next, let $\widetilde{X}_L := \Spec(\CC[\widetilde{\OO}_L])$. This is a normal, irreducible, affine variety with an action of $L$, an open embedding $\widetilde{\OO}_L \subset \widetilde{X}_L$, and a finite $L$-equivariant surjection $\mu_L:\widetilde{X}_L \to \overline{\OO}_L$. Varfieties of this form have been studied extensively in \cite{LMBM}. Let $P$ act on the variety $\{\xi\} \times \widetilde{X}_L \times \fp^{\perp}$ in the following manner: let $L$ act diagonally and $N$ by the formula
$$n \cdot (\xi,x,y)=(\xi,x,n\mu_L(x)-\mu_L(x)+ny),\quad n \in N, \quad x\in \widetilde{X}_L, \quad y\in \fp^\perp,$$
By our construction of the $P$-action on $\{\xi\} \times \widetilde{X}_L \times \fp^{\perp}$, there is a $P$-equivariant map $\{\xi\} \times \widetilde{X}_L \times \fp^{\perp} \to \xi+\overline{\OO}_L + \fp^{\perp}$ and hence a $G$-equivariant map $G \times^P (\{\xi\} \times \widetilde{X}_L \times \fp^{\perp}) \to G \times^P (\xi + \overline{\OO}_L + \fp^{\perp})$ of fiber bundles over $G/P$. Consider the composition
$$\widetilde{\mu}: G \times^P (\{\xi\} \times \widetilde{X}_L \times \mathfrak{p}^{\perp}) \to G \times^P (\xi+\overline{\mathbb{O}}_L + \mathfrak{p}^{\perp}) \overset{\mu}{\to} \fg^*.$$
The image of $\widetilde{\mu}$ is the closure of $\Ind (L,\OO_L,\xi)$ and the restriction of $\widetilde{\mu}$ to the preimage of $\Ind (L,\OO_L,\xi)$ is a (finite, connected, $G$-equivariant) cover, denoted $\mathrm{Bind} (L,\widetilde{\OO}_L,\xi)$. We note that the co-adjoint orbit $\Ind (L,\OO_L,\xi)$ and its cover $\mathrm{Bind} (L,\widetilde{\OO}_L,\xi)$ are both independent of the choice of parabolic $P$, see \cite[Lemma 4.1]{Losev4} and \cite[Corollary 4.3]{Mitya2020}. Moreover, conjugate birational induction data give rise to isomorphic covers. Thus, we obtain a well-defined map
\begin{equation}\label{eq:Bind}\mathrm{Bind}: \Omega(G) \to \mathsf{Cov}(G)\end{equation}
It is easy to see that this map is surjective.

A co-adjoint cover is \emph{birationally rigid} if it cannot be obtained via birational induction from a proper Levi subgroup (if a cover is birationally rigid, it must be nilpotent). Accordingly, a birational induction datum $(L,\widetilde{\OO}_L,\xi)$ is \emph{minimal} if $\widetilde{\OO}_L$ is birationally rigid. Note that the set of minimal birational induction data is stable under conjugation by $G$. Write $\Omega_m(G) \subset \Omega(G)$ for the subset of $G$-conjugacy classes of minimal birational induction data. The following is contained in \cite[Theorem 7.8.1]{LMBM} (where the proof is adapted from \cite[Corollary 4.6]{Losev4}). 

\begin{prop}\label{prop:bind}
The map (\ref{eq:Bind}) restricts to a bijection
$$\mathrm{Bind}: \Omega_m(G) \xrightarrow{\sim} \mathsf{Cov}(G)$$
\end{prop}

\subsection{Unitary characters and integral co-adjoint covers}\label{sec:integrality}

Let $\widehat{G}_u$ denote the abelian group of unitary characters of $G$ (i.e. continuous homomorphisms $G \to S^1$). Choose a maximal torus $H \subset G$ and let $\mathfrak{h}$ be its Lie algebra. Use the symbol $\overline{\bullet}$ to denote complex conjugation on $\fh$ and on $\fh^*$. There is an identification (differentiation at the identity)
\begin{equation}\label{eq:id1}\widehat{G}_u \xrightarrow{\sim} \{\lambda \in \Hom_{\RR}(\fz(\fg),i\RR) \mid \frac{1}{2}(\lambda+\overline{\lambda}) \in X^*(H)\} =: \Hom_{\RR}(\fz(\fg),i\RR)_{int}\end{equation}
Using the isomorphism (\ref{eq:iota}) (applied to the vector space $\fz(\fg)$), we get a further identification
\begin{equation}\label{eq:id2}\widehat{G}_u \xrightarrow{\sim} \{\mu \in \fz(\fg)^* \mid \frac{1}{2}(\iota(\mu) + \overline{\iota(\mu)}) \in X^*(H)\} =: \fz(\fg)^*_{int}\end{equation}
The Langlands classification (for the abelian complex group $G/[G,G]$) provides a third isomorphism
\begin{equation}\label{eq:id3}\widehat{G}_u \xrightarrow{\sim} \{\text{bounded continuous homomorphisms } \CC^{\times} \to Z(G^{\vee})^{\circ}\} =: \Phi_{\mathrm{temp}}(Z(G^{\vee})^{\circ})\end{equation}
We are now prepared to define the `integrality condition' of Conjecture \ref{conj:orbit}.

\begin{definition}\label{def:minimaldatum}\ \\
\begin{itemize}
    \item A minimal birational induction datum $(L,\widetilde{\OO}_L,\xi)$ is \emph{integral} if 
$$\xi \in \fz(\fl)^*_{int},$$
see (\ref{eq:id2}). The set of integral minimal birational induction is stable under conjugation. Write $\Omega_{m,int}(G) \subset \Omega_m(G)$ for the set of $G$-conjugacy classes of integral minimal birational induction data. 
\item A complex co-adjoint cover $\widetilde{\OO} \in \mathsf{Cov}(G)$ is \emph{integral} if it corresponds under $\mathrm{Bind}$ to an integral minimal birational induction datum, see Proposition \ref{prop:bind}. Write $\mathsf{Cov}_{int}(G)$ for the set of (isomorphism classes of) integral complex co-adjoint covers.
\item A real co-adjoint cover $\widetilde{\OO} \in \mathsf{Cov}^{i\RR}(G)$ is \emph{integral} if it corresponds under $\iota$ (see Lemma \ref{lem:imaginary}) to an integral complex co-adjoint cover. Write $\mathsf{Cov}^{i\RR}_{int}(G)$ for the set of (isomorphism classes of) integral real co-adjoint covers.
\end{itemize}
\end{definition}

\subsection{Unipotent representations}

Let $\widetilde{\OO}$ be a complex nilpotent cover. In \cite{LMBM}, we construct a finite set of irreducible admissible $G$-representations
$$\mathrm{Unip}_{\widetilde{\OO}}(G) \subset \Pi(G)$$
The constituents of $\mathrm{Unip}_{\widetilde{\OO}}(G)$ are called \emph{unipotent representations} (this presents a small problem of terminology. These unipotent representations are a priori distinct from the representations attached to unipotent Arthur parameters (Definition \ref{def:Arthur}). However, we will see below that the latter set is contained in the former set, so we will not bother to make our terms more precise). We will now briefly review the construction of these sets and some of their essential properties.

The group $\CC^{\times}$ acts by dilation on $\OO \subset \fg^*$. Suitably normalized, this action can be lifted to the co-adjoint cover $\widetilde{\OO}$. There is also a natural symplectic structure on $\OO$ (defined using the Lie bracket). This too can be lifted to $\widetilde{\OO}$. Thus the ring of regular functions $\CC[\widetilde{\OO}]$ is a graded Poisson algebra. Accordingly, it is reasonable to talk about \emph{filtered quantizations} of $\CC[\widetilde{\OO}]$. Formally, these are pairs $(\cA,\theta)$ consisting of a an associative algebra $\cA$ equipped with an increasing filtration $\cA = \bigcup_i \cA_{\leq i}$ such that $[\cA_{\leq i},\cA_{\leq j}] \subseteq \cA_{\leq i+j-1}$ and an isomorphism of graded Poisson algebras $\theta: \gr(\cA) \xrightarrow{\sim} \CC[\widetilde{\OO}]$ (the Poisson bracket on $\gr(\cA)$ is defined by $\{x+\cA_{\leq i-1},y + \cA_{\leq j-1}\} = [x,y] + \cA_{i+j-2}$). A general discussion of filtered quantizations (relevant to the study of co-adjoint covers) can be found in \cite[Section 4]{LMBM}. 

Write $\mathrm{Quant}(\CC[\widetilde{\OO}])$ for the set of isomorphism classes of filtered quantizations of $\CC[\widetilde{\OO}]$. For an arbitrary graded Poisson algebra, filtered quantizations may be impossible to classify. But the algebra $\CC[\widetilde{\OO}]$ has a very special property which can be used to our advantage: its spectrum $\Spec(\CC[\widetilde{\OO}])$ has \emph{symplectic singularities} in the sense of Beauville \cite[Definition 1.1]{Beauville2000}. Because of this fact, we can apply the deep results of Losev \cite{Losev4} to parameterize the set $\mathrm{Quant}(\CC[\widetilde{\OO}])$. The conclusion is that there is a complex vector space $\fP$ (recovered from the geometry of $\Spec(\CC[\widetilde{\OO}])$), a finite reflection group $\mathcal{W} \subset \mathrm{GL}(\fP)$, and a canonical bijection
$$\fP/\mathcal{W} \xrightarrow{\sim} \mathrm{Quant}(\CC[\widetilde{\OO}])$$
This is \cite[Proposition 3.3, Theorem 3.4]{Losev4}. 

The \emph{canonical quantization} of $\CC[\widetilde{\OO}]$ is the filtered quantization $(\cA_0,\theta_0)$ corresponding to the point $0 \in \fP$. The $G$-action on $\CC[\widetilde{\OO}]$ lifts to a $G$-action on $\cA_0$ and the classical co-moment map $\CC[\fg^*] \to \CC[\widetilde{\OO}]$ (induced by the map $\widetilde{\OO} \to \fg^*$) lifts to a $G$-equivariant filtered algebra homomorphism $\Phi_0: U(\fg) \to \cA_0$ from the universal enveloping algebra $U(\fg)$ of $\fg$ (the first fact is automatic since $G$ is reductive, the second fact is \cite[Lemma 4.11.2]{LMBM}). 

\begin{definition}[Definition 6.0.1, \cite{LMBM}]
The \emph{unipotent ideal} attached to $\widetilde{\OO}$ is the two-sided ideal
$$I_0(\widetilde{\OO}) := \ker{(\Phi_0: U(\fg) \to \cA_0)}$$
\end{definition}

The unipotent ideal $I_0(\widetilde{\OO})$ is a completely-prime primitive ideal with associated variety $\overline{\OO}$ \cite[Proposition 6.1.2]{LMBM}. In fact, it is a \emph{maximal} ideal in $U(\fg)$ \cite[Theorem 5.0.1]{MBMat}. Recall that a $U(\fg)$-bimodule $M$ is \emph{Harish-Chandra} if the adjoint action of $\fg$
$$\ad(X)m = Xm-mX, \qquad X \in \fg, \ m \in M$$
integrates to an algebraic action of $G$. There is a well-known bijection between irreducible admissible $G$-representations and irreducible Harish-Chandra bimodules (this is a special case of a deep result of Harish-Chandra for real reductive Lie groups. The special case was first observed in \cite{Duflo1977}).

\begin{definition}[Definition 6.0.2, \cite{LMBM}]\label{def:unipotent} \ \\
\begin{itemize}
    \item A \emph{unipotent Harish-Chandra bimodule} attached to $\widetilde{\OO}$ is an irreducible Harish-Chandra bimodule which is annihilated on both sides by the unipotent ideal $I_0(\widetilde{\OO})$.
    \item A \emph{unipotent $G$-representation} attached to $\widetilde{\OO}$ is an irreducible admissible $G$-representation which corresponds to a unipotent Harish-Chandra bimodule attached to $\widetilde{\OO}$. 
    \item $\mathrm{Unip}_{\widetilde{\OO}}(G)$ is the (finite) set of equivalence classes of unipotent $G$-representations attached to $\widetilde{\OO}$. 
\end{itemize}
\end{definition}

In \cite{LMBM} and \cite{MBMat}, we establish a number of properties of unipotent representations. For example, we parameterize the sets $\mathrm{Unip}_{\widetilde{\OO}}(G)$ in terms of geometric data and compute the infinitesimal characters in all cases. We also prove, when $G = \mathrm{SL}(n)$, $\mathrm{SO}(n)$, or $\mathrm{Sp}(2n)$, that the constituents of $\mathrm{Unip}_{\widetilde{\OO}}(G)$ are unitary. We expect this to remain true in general, but we do not have a proof. 

\subsection{Kirillov packets}

We are now prepared to define the Kirillov packets $\Pi^{\mathsf{Kir}}_{\widetilde{\OO}}(G)$ for a complex group $G$. The basic idea is already contained in the writing of Vogan on the orbit method (see, \cite{Vogan1992}, \cite{Voganorbit}): use parabolic induction to reduce to the case when $\widetilde{\OO}$ is nilpotent. When $\widetilde{\OO}$ is nilpotent, we take $\Pi^{\mathsf{Kir}}_{\widetilde{\OO}}(G) = \mathrm{Unip}_{\widetilde{\OO}}(G)$ (the significant new idea, then, is a definition of $\mathrm{Unip}_{\widetilde{\OO}}(G)$). 

\begin{definition}\label{def:Kirpackets}
Let $\widetilde{\OO} \in \mathsf{Cov}^{i\RR}_{int}(G)$. Then $\widetilde{\OO}$ corresponds to an integral complex co-adjoint cover $\widetilde{\OO}' \in \mathsf{Cov}_{int}(G)$ under the bijection (\ref{eq:iotacovers}). Choose $(L,\widetilde{\OO}_L,\xi) \in \Omega_{m,int}(G)$ such that $\widetilde{\OO}' = \mathrm{Bind} (L, \widetilde{\OO}_L, \xi)$. By Proposition \ref{prop:bind}, $(L,\widetilde{\OO}_L,\xi)$ is unique up to conjugation by $G$. The element $\xi \in \fz(\fl)^*_{int}$ corresponds to an element $\iota(\xi) \in \Hom_{\RR}(\fz(\fl),i\RR)_{int}$ under the isomorphism (\ref{eq:iota}) (applied to the vector space $\fz(\fl)$) which in turn corresponds to a unitary character of $L$ under the bijection (\ref{eq:id1}) (applied to the group $L$). Denote this unitary character also by $\iota(\xi)$. The \emph{Kirillov packet} attached to $\widetilde{\OO}$ is the finite set of irreducible representations
\begin{align*}
\Pi^{\mathsf{Kir}}_{\widetilde{\OO}}(G) := \{X \in \Pi(G) \mid &\exists \ X_0 \in \mathrm{Unip}_{\widetilde{\OO}_L}(L) \text{ such that }\\
&X \text{ is a composition factor in } \Ind^G_P (X_0 \otimes \iota(\xi))\}\end{align*}
\end{definition}

\begin{rmk}
As stated in the remarks following Definition \ref{def:unipotent}, we believe that the constituents of $\mathrm{Unip}_{\widetilde{\OO}_L}(L)$ should be unitary (and we have proved this in many cases). Under this assumption, the induced representations $ \Ind^G_P (X_0 \otimes \iota(\xi))$ appearing in Definition \ref{def:Kirpackets} are unitary and completely reducible. So the constituents of $\Pi^{\mathsf{Kir}}_{\widetilde{\OO}}(G)$ are unitary as well.
\end{rmk}

\subsection{Duality}\label{sec:duality}

In this section, we will relate Arthur packets (Definition \ref{def:Arthur}) to Kirillov packets (Definition \ref{def:Kirpackets}) for complex groups. More precisely, we will construct a natural injective map from Arthur parameters $\Psi(G^{\vee})$ to integral real co-adjoint covers $\mathsf{Cov}^{i\RR}_{int}(G)$ such that the Arthur packet $\Pi_{\psi}^{\mathsf{Art}}(G)$ corresponding to an Arthur parameter $\psi$ coincides with the Kirillov packet $\Pi^{\mathsf{Kir}}_{\mathsf{D}(\psi)}(G)$ corresponding to the co-adjoint cover $\mathsf{D}(\psi)$. The main difficulty is defining such a map on \emph{unipotent} Arthur parameters. This was essentially done in \cite[Section 9]{LMBM}. 

Suppose $\psi \in \Psi(G^{\vee})$. Since $G$ is complex, we can think of $\psi$ as a continuous homomorphism $\CC^{\times} \times \mathrm{SL}(2,\CC) \to G^{\vee}$. Let $L^{\vee}$ denote the centralizer in $G^{\vee}$ of $\psi(\CC^{\times})$. Note that $L^{\vee}$ is a Levi subgroup of $G^{\vee}$ and hence the dual group of a Levi subgroup $L$ of $G$. Define a unipotent Arthur parameter $\psi_0 \in \Psi(L^{\vee})$
$$\psi_0: \CC^{\times} \times \mathrm{SL}(2,\CC) \to L^{\vee}, \qquad \psi_0(z,g) = \psi(g)$$
and a tempered Langlands parameter $\psi_1 \in \Phi_{\mathrm{temp}}(Z(L^{\vee})^{\circ})$
$$\psi_1: \CC^{\times} \to Z(L^{\vee})^{\circ}, \qquad \psi_1(z) = \psi(z,1).$$
In \cite[Section 9]{LMBM}, we define a complex nilpotent cover $\mathsf{D}_0(\psi_0) \in \mathsf{Cov}(L)$ such that
$$\Pi^{\mathsf{Art}}_{\psi_0}(L) = \mathrm{Unip}_{\mathsf{D}_0(\psi_0)}(L)$$
The tempered Langlands parameter $\psi_1$ corresponds via (\ref{eq:id3}) to a unitary character $\chi_1$ of $L$ and then via (\ref{eq:id2}) to an element $\xi_1 \in \fz(\fl)^*_{int}$. Consider the birational induction datum
$$(L,\mathsf{D}_0(\psi_0),\xi_1) \in \Omega(G)$$
and the complex co-adjoint cover
$$\mathrm{Bind} (L,\mathsf{D}_0(\psi_0),\xi_1) \in \mathsf{Cov}(G)$$
Since $\mathsf{D}_0(\psi_0)$ is nilpotent and $\xi_1$ is integral, the cover above is integral (i.e. contained in the set $\mathsf{Cov}_{int}(G)$). Thus, it corresponds via $\iota$ to an integral real co-adjoint cover, which we denote by $\mathsf{D}(\psi)$
$$\mathsf{D}(\psi) := \iota(\mathrm{Bind}(L,\mathsf{D}_0(\psi_0),\xi_1) \in \mathsf{Cov}^{i\RR}_{int}(G)$$
This defines a map
\begin{equation}\label{eq:duality}\mathsf{D}: \Psi(G) \to \mathsf{Cov}^{i\RR}_{int}(G)\end{equation}
\begin{theorem}\label{thm:duality}
The map (\ref{eq:duality}) has the following properties:
\begin{itemize}
    \item[(i)] $\mathsf{D}$ is injective.
    \item[(ii)] $\Pi^{\mathsf{Art}}_{\psi}(G) = \Pi^{\mathrm{Kir}}_{\mathsf{D}(\psi)}(G)$ for all $\psi \in \Psi(G^{\vee})$.
    \item[(iii)] If $\psi$ is a unipotent Arthur parameter, then $\mathsf{D}(\psi)$ is a nilpotent cover.
\end{itemize}
\end{theorem}

\begin{proof}[Proof sketch]
(i) follows immediately from the injectivity of $\mathsf{D}_0: \Psi_{unip}(L^{\vee}) \to \mathsf{Cov}(L)$ (\cite[Proposition 9.2.1(iii)]{LMBM}) and the injectivity of $\mathrm{Bind}: \Omega_m(G) \to \mathsf{Cov}(G)$ (Proposition \ref{prop:bind}). (iii) is immediate from the definition of $\mathsf{D}$. It remains to prove (ii). Since (ii) holds for all \emph{unipotent} Arthur parameters, it suffices to show that an arbitrary Arthur packet $\Pi^{\mathsf{Art}}_{\psi}(G)$ has a description analogous to Definition \ref{def:Kirpackets}, i.e.
\begin{align*}
\Pi^{\mathsf{Art}}_{\psi}(G) = \{X \in \Pi(G) \mid \ &\exists \  X_0 \in \Pi_{\psi_0}^{\mathsf{Art}}(L) \text{ such that }\\
&X \text{ is a composition factor in } \Ind^G_P (X_0 \otimes \chi_1))\}\end{align*}
This follows from the results of \cite[Chapter 26]{AdamsBarbaschVogan}.
\end{proof}

This result shows that all of Arthur's representations for complex groups arise via the orbit method. We note that the map $\mathsf{D}$ is not in general surjective, i.e. there are many (unitary) representations of $G$ not considered by Arthur which arise via the orbit method (for example, the metaplectic representations of $\mathrm{Sp}(2n,\CC)$).

\subsection{The general case}

In this section, we will sketch a generalization of our complex Orbit Method (Definition \ref{def:Kirpackets}) for arbitrary real groups. 

Let $G_{\RR}$ be the real points of a connected reductive algebraic group $G$. Choose a Cartan involution $\theta: G \to G$ and let $K \subset G$ denote the fixed points of $\theta$ (so that $K \cap G_{\RR} \subset G_{\RR}$ is a maximal compact subgroup). By a result of Harish-Chandra, the category of finite-length admissible representations of $G_{\RR}$ is naturally equivalent to the category $M(\fg,K)$ of finite-length $(\fg,K)$-modules. In particular, the set $\Pi(G_{\RR})$ of equivalence classes of irreducible admissible $G_{\RR}$-representations is in bijection with the set of equivalence classes of irreducible $(\fg,K)$-modules (in practice, we will often identify these two sets).

Let $\cN \subset \fg^*$ denote the nilpotent cone, and let $\fp \subset \fg$ denote the $-1$-eigenspace of the involution $d\theta$ of $\fg$. Then $K$ acts on $\cN_{\theta} := \cN \cap \fp^*$ with finitely many orbits (\cite[Corollary 5.22]{Vogan1991}), and each $K$-orbit is a Lagrangian subvariety of its $G$-saturation (\cite[Proposition 5.19]{Vogan1991}). Write $\mathsf{Orb}^{\theta}_{nilp}(G)$ for the set of $K$-orbits on $\cN_{\theta}$ and $\mathsf{Cov}^{\theta}_{nilp}(G)$ for the set of isomorphism classes of finite connected $K$-equivariant covers of $K$-orbits on $\cN_{\theta}$. The following result is essentially due to Kostant (\cite{KostantRallis1971}) and Sekiguchi (\cite{Sekiguchi1987}). The formulation below comes from \cite[Section 6]{Vogan1991}.

\begin{prop}\label{prop:KostantSekiguchi}
There is a natural bijection
\begin{equation}\label{eq:KostantSekiguchi1}\iota: \mathsf{Orb}^{\theta}_{nilp}(G) \xrightarrow{\sim} \mathsf{Orb}^{i\RR}_{nilp}(G_{\RR})\end{equation}
such that
$$G\OO = G\iota(\OO), \qquad \OO \in  \mathsf{Orb}^{\theta}_{nilp}(G).$$
For each $\OO \in  \mathsf{Orb}^{\theta}_{nilp}(G)$, the $K$-equivariant fundamental group of $\OO$ is naturally identified with the $G_{\RR}$-equivariant fundamental group of $\iota(\OO)$. So (\ref{eq:KostantSekiguchi1}) lifts to a bijection
\begin{equation}\label{eq:kostantSekiguchi2}
    \iota: \mathsf{Cov}^{\theta}_{nilp}(G) \xrightarrow{\sim} \mathsf{Cov}^{i\RR}_{nilp}(G_{\RR})
\end{equation}
\end{prop}

To each irreducible irreducible $(\fg,K)$-module $X$, one can attach a closed $K$-invariant subset $V(X) \subset \cN_{\theta}$ called the \emph{associated variety} of $X$, see \cite[Corollary 5.13]{Vogan1991}. In general, $V(X)$ may not be irreducible (although it is always equidimensional).

The first step towards defining an Orbit Method for $G_{\RR}$ is to define a suitable set of `building blocks'. 

\begin{definition}\label{def:unipotentreal}
Let $\widetilde{\OO}$ be a birationally rigid nilpotent cover for $G$ and let $\OO_{\theta}$ be a $K$-orbit on $\OO \cap \mathfrak{p}^*$. A \emph{unipotent representation} attached to the pair $(\widetilde{\OO}, \OO_{\theta})$ is an irreducible $(\fg,K)$-module $X$ such that
\begin{itemize}
    \item[(i)] $\Ann_{U(\fg)}(X) = I(\widetilde{\OO})$.
    \item[(ii)] $V(X) = \overline{\OO}_{\theta}$.
\end{itemize}
Write $\mathrm{Unip}_{(\widetilde{\OO},\OO_{\theta})}(G_{\RR})$ for the (finite) set of equivalence classes of unipotent representations attached to $(\widetilde{\OO},\OO_{\theta})$. 
\end{definition}

The nilpotent covers $\widetilde{\OO}$ and unipotent ideals $I(\widetilde{\OO})$ which appear in Definition \ref{def:unipotentreal} were computed in all cases in \cite{LMBM} and \cite{MBMat}. For our construction to make sense, the following should be verified.

\begin{conj}\label{conj:unitarity}
The following are true:
\begin{itemize}
\item[(i)] For every pair $(\widetilde{\OO},\OO_{\theta})$ as in Definition \ref{def:unipotentreal}, the 
constituents of $\mathrm{Unip}_{(\widetilde{\OO},\OO_{\theta})}(G_{\RR})$ are unitary.
\item[(ii)] The unipotent ideals $I(\widetilde{\OO})$ appearing in Definition \ref{def:unipotentreal} are \emph{weakly unipotent} (cf. \cite[Definition 12.3]{KnappVogan1995}). 
\end{itemize}
\end{conj}

For $G$ simple exceptional and $\widetilde{\OO}$ an orbit, Conjecture \ref{conj:unitarity}(i) was verified in \cite[Theorem 6.1.1]{MBMat}. Conjecture \ref{conj:unitarity}(ii) should be very easy to check using the computations in \cite{MBMat}.

We will now introduce the appropriate generalization of Definition \ref{def:Kirpackets}. This will require several small digressions. Suppose $L$ is a $\theta$-stable Levi subgroup of $G$. Let
$$L_c := L^{\theta}, \qquad \mathfrak{l}_n := \mathfrak{l}^{-d\theta}.$$
and similarly
$$Z_c(L) := Z(L)^{\theta}, \qquad \mathfrak{z}_n(\fl) := \fz(\fl)^{-d\theta}$$
Consider the Levi subgroup
$$L_1 := Z_G(\fz_n(\fl)) \subset G$$
Since $\fz_n(\fl)$ is $\theta$-stable, $L_1$ is $\theta$-stable. Write
$$L_{1,c} = L_1^{\theta}, \qquad L_{1,n} = \fl_1^{-d\theta}.$$
There is an inclusion $L \subseteq L_1$ and $L = Z_{L_1}(\fz_c(\fl))$. Choose a Cartan subalgebra $\fh \subset \fl$ and let $\Delta := \Delta(\fh,\fg)$, the roots of $\fh$ on $\fg$. Then 
$$\Delta_1:=\Delta(\fh,\fl_1) = \{\alpha \in \Delta \mid \alpha(\fz_n(\fl)) = 0\}$$
If we choose a regular element $t \in \fz_c(\fl)^*$ in the real span of $\Delta_1$, we get a parabolic subalgebra $\mathfrak{q}  \subset \mathfrak{l}_1$ defined by
$$\mathfrak{q}=\mathfrak{l}\oplus \mathfrak{u}, \qquad \mathfrak{u} = \bigoplus_{\alpha \in \Delta_1, \alpha(t) >0} \fg_{\alpha}$$
Since $\theta(t)=t$, we have that $\theta \mathfrak{q} = \mathfrak{q}$.

Similarly, if we choose a regular element $a \in \fz_n(\fl)^*$ in the real span of $\Delta$, we get a parabolic subalgebra
$$\mathfrak{q}_1 = \mathfrak{l}_1 \oplus \mathfrak{u}_1, \qquad \mathfrak{u}_1 = \bigoplus_{\alpha \in \Delta, \alpha(a) >0} \fg_{\alpha}$$
Since $\theta(a) = -a$, we have that $\theta\mathfrak{q}_1 = \mathfrak{q}_1^{\mathrm{op}}$ (the opposite parabolic).

Write $2\rho(\mathfrak{u})$ for the determinant character of $L$ on $\mathfrak{u}$
$$2\rho(\mathfrak{u})(l) = \det(\Ad(l)|_{\mathfrak{u}}), \qquad l \in L.$$
Form the fiber product
\begin{center}
    \begin{tikzcd}
    \widetilde{L} \ar[r] \ar[d] & \CC^{\times} \ar[d,"(\bullet)^2"]\\
    L \ar[r,"2\rho(\mathfrak{u})"] & \CC^{\times}
    \end{tikzcd}
\end{center}
The left vertical map $\widetilde{L} \to L$ is a 2-fold cover of $L$ and the top horizontal map $\widetilde{L} \to \CC^{\times}$ is a genuine character of $\widetilde{L}$, which we will denote by $\rho(\mathfrak{u})$. It is not hard to see that the cover $\widetilde{L}$ of $L$ and the character $\rho(\mathfrak{u})$ of $\widetilde{L}$ are independent (up to canonical isomorphisms) of the choice of parabolic $\mathfrak{q}$, see the remarks following \cite[Proposition 1.35]{Vogan1987}. We will write $\widetilde{L}_c$ for the preimage of $L_c$ under the homomorphism $\widetilde{L} \to L$.

\begin{definition}\label{def:inductiondatumreal}
An integral minimal birational induction datum is a quintuple
$$(L,\widetilde{\OO}_L,\OO_{L,\theta},\mathfrak{q},\chi)$$
consisting of
\begin{itemize}
    \item A $\theta$-stable Levi subgroup $L \subset G$.
    \item A birationally rigid cover $\widetilde{\OO}_L$ of a nilpotent co-adjoint orbit $\OO_L \subset \fl^*$.
    \item An $L_c$-orbit $\OO_{L,\theta}$ on $\OO_L \cap \mathfrak{l}_n^*$.
    \item A $\theta$-stable parabolic $\mathfrak{q}=\mathfrak{l}\oplus \mathfrak{u}$ in $\mathfrak{l}_1$.
    \item A genuine one-dimensional unitary Harish-Chandra $(\fl,\widetilde{L}_c)$-module $\chi$.
\end{itemize}
subject to the condition
\begin{equation}\label{eq:dominance}\langle d\chi, \alpha^{\vee} \rangle \geq 0, \qquad \forall \alpha \in \Delta_1^+.\end{equation}
Write $\Omega_{m,int}(G_{\RR})$ for the set $K$-conjugacy classes of integral minimal birational induction data. 
\end{definition}

Given an integral minimal birational induction datum $\Gamma = (L,\widetilde{\OO}_L,\OO_{L,\theta},\mathfrak{q},\chi)$, we will construct a finite packet $\Pi_{\Gamma}^{\mathsf{Kir}}(G_{\RR})$ of irreducible (conjecturally unitary) representations of $G_{\RR}$. The construction proceeds in stages. There are cohomological induction functors
$$(\mathcal{R}^{\fl_1,L_{1,c}}_{\fq,\widetilde{L}_c})^j: M(\fl,\widetilde{L}_c) \to M(\fl_1,L_{1,c}), \qquad j \in \ZZ_{\geq 0}$$
see \cite[Theorem 6.8]{Vogan1987}. These functors preserve infinitesimal character. For each $X_0 \in \mathrm{Unip}_{(\widetilde{\OO}_L,\OO_{L,\theta})}(L_{\RR})$, consider
$$(\mathcal{R}^{\fl_1,L_{1,c}}_{\fq,\widetilde{L}_c})^j(\chi \otimes X_0) \in M(\fl_1,L_{1,c}), \qquad j \in \ZZ_{\geq 0}$$
Because of the dominance condition on $\chi$ (\ref{eq:dominance}), the inducing module $\chi \otimes X_0$ is in the \emph{weakly fair range} (with respect to the parabolic $\mathfrak{q}$). So if we assume Conjecture \ref{conj:unitarity}, then 
\begin{itemize}
    \item $(\mathcal{R}^{\fl_1,L_{1,c}}_{\fq,\widetilde{L}_c})^j(\chi \otimes X_0)=0$ if $j \neq d(\Gamma)$, where
$$d(\Gamma) := \dim(\mathfrak{u} \cap \mathfrak{k}).$$
    \item $(\mathcal{R}^{\fl_1,L_{1,c}}_{\fq,\widetilde{L}_c})^{d(\Gamma)}(\chi \otimes X_0)$ is unitary.
\end{itemize}
This is \cite[Theorem 12.4]{KnappVogan1995}.

Next, choose a parabolic $Q_1 = L_1U_1$ of $G$ such that $\theta Q_1 = Q_1^{\mathrm{op}}$, see the remarks preceding Definition \ref{def:Kirpacketsreal}. Conjugating by $K$, we can assume that $Q_1$ is defined over $\RR$. There is a (normalized) real parabolic induction functor
$$\Ind^{G_{\RR}}_{Q_{1,\RR}}: M(\fl_1,L_{1,c}) \to M(\fg,K).$$
This functor is exact. It preserves infinitesimal character and unitarity. 

\begin{definition}\label{def:Kirpacketsreal}
Suppose $\Gamma=(L,\widetilde{\OO}_L,\OO_{L,\theta}, \mathfrak{q},\chi)$ is an integral minimal birational induction datum (cf. Definition \ref{def:inductiondatumreal}). The \emph{Kirillov packet} attached to $\Gamma$ is the finite set of irreducible $(\fg,K)$-modules
\begin{align*}
    \Pi_{\Gamma}^{\mathsf{Kir}}(G_{\RR}) := \{X \in \Pi(G_{\RR}) \mid \ &\exists \ X_0 \in \mathrm{Unip}_{(\widetilde{\OO}_L,\OO_{L,\theta})}(L_{\RR}) \text{ such that}\\
    &X \text{ is a composition factor in } \Ind^{G_{\RR}}_{L_{1,\RR}}(\mathcal{R}^{\fl_1,L_{1,c}}_{\fq,\widetilde{L}_c})^{d(\Gamma)} (X_0 \otimes \chi)\}
\end{align*}
If $d\chi=0$, the constitutents of the packet  $\Pi_{\Gamma}^{\mathsf{Kir}}(G_{\RR})$ are called \emph{unipotent representations}.
\end{definition}

\begin{rmk}
If we assume Conjecture \ref{conj:unitarity} is true, then the induced modules
$$\Ind^{G_{\RR}}_{L_{1,\RR}}(\mathcal{R}^{\fl_1,L_{1,c}}_{\fq,\widetilde{L}_c})^{d(\Gamma)} (X_0 \otimes \chi)$$
appearing in the definition above are unitary. Hence the representations in $\Pi_{\Gamma}^{\mathsf{Kir}}(G_{\RR})$ are unitary as well (and the phrase `composition factor' can be replaced with `direct summand').
\end{rmk}

\begin{rmk}
It is worth highlighting a few special cases of Definition \ref{def:Kirpacketsreal}.
\begin{itemize}
    \item Assume $L$ is a maximal torus in $K$. Then $\widetilde{\OO}_L = \{0\}$, $\OO_{L,\theta}=\{0\}$, $\mathfrak{q}$ is a Borel subalgebra of $\mathfrak{g}$, and $\chi$ is a ($\rho$-shifted) algebraic character of $L$. The packet $\Pi_{\Gamma}^{\mathsf{Kir}}(G_{\RR})$ is a singleton consisting of the unique (limit) of discrete series representation with Harish-chandra parameter $(\mathfrak{q},\chi)$.
    \item Assume $L$ is the complexification of a split maximal torus. Then $\widetilde{\OO}_L = \{0\}$, $\OO_{L,\theta}=\{0\}$, $\mathfrak{q} = \fh$, and $\chi$ is a unitary character of $L_{\RR}$. The packet $\Pi_{\Gamma}^{\mathsf{Kir}}(G_{\RR})$ is the set of direct summands of the unitary principal series representation $\Ind^{G_{\RR}}_{B_{\RR}}\chi$.
    
    \item Assume $G_{\RR}$ is a complex connected reductive algebraic group, regarded as a real group by restriction of scalars. Then Definition \ref{def:Kirpacketsreal} is equivalent to Definition \ref{def:Kirpackets}.
\end{itemize} 

\end{rmk}

In the form given above, it is not clear how (if at all) Definition \ref{def:Kirpacketsreal} qualifies as an `Orbit Method'. We have defined the Kirillov packets but not the co-adjoint covers to which they are attached. What is needed is a map $\mathcal{V}$ which assigns a a co-adjoint cover (or collection thereof) to each induction datum $\Gamma \in \Omega_{m,int}(G_{\RR})$. If $\Gamma$ is unipotent (i.e. if $d\chi=0$), then $\mathcal{V}(\Gamma)$ should correspond (via Proposition \ref{prop:KostantSekiguchi}) to the associated variety of the modules $\Ind^{G_{\RR}}_{L_{1,\RR}}(\mathcal{R}^{\fl_1,L_{1,c}}_{\fq,\widetilde{L}_c})^{d(\Gamma)} (X_0 \otimes \chi)$. By analyzing the effect of real parabolic and cohomological induction on associated varieties, it should not be difficult to define such a map, but we will not do so here.

\printbibliography
\end{document}